\pdfoutput=1
\documentclass[reqno, 12pt]{amsart}
\def\ssign{\textsection\nobreak\hspace{1pt plus 0.3pt}}
\makeatletter
\let\origsection=\section 
\def\mysection{\@mystartsection{section}{1}\z@{.7\linespacing\@plus\linespacing}{.5\linespacing}{\normalfont\scshape\centering\ssign}}
\def\section{\@ifstar{\origsection*}{\mysection}}
\def\appendix{\par\c@section\z@ \c@subsection\z@
   \let\sectionname\appendixname
   \let\section=\origsection 
   \def\thesection{\@Alph\c@section}} 
\def\@mystartsection#1#2#3#4#5#6{\if@noskipsec \leavevmode \fi 
 \par \@tempskipa #4\relax
 \@afterindenttrue
 \ifdim \@tempskipa <\z@ \@tempskipa -\@tempskipa \@afterindentfalse\fi
 \if@nobreak \everypar{}\else
     \addpenalty\@secpenalty\addvspace\@tempskipa\fi
 \@dblarg{\@mysect{#1}{#2}{#3}{#4}{#5}{#6}}}
\def\@mysect#1#2#3#4#5#6[#7]#8{\edef\@toclevel{\ifnum#2=\@m 0\else\number#2\fi}\ifnum #2>\c@secnumdepth \let\@secnumber\@empty
  \else \@xp\let\@xp\@secnumber\csname the#1\endcsname\fi
  \@tempskipa #5\relax
  \ifnum #2>\c@secnumdepth
    \let\@svsec\@empty
  \else
   ~\refstepcounter{#1}\edef\@secnumpunct{\ifdim\@tempskipa>\z@ \@ifnotempty{#8}{\@nx\enspace}\else
        \@ifempty{#8}{.}{\@nx\enspace}\fi
    }\@ifempty{#8}{\ifnum #2=\tw@ \def\@secnumfont{\bfseries}\fi}{}\protected@edef\@svsec{\ifnum#2<\@m
        \@ifundefined{#1name}{}{\ignorespaces\csname #1name\endcsname\space
        }\fi
      \@seccntformat{#1}}\fi
  \ifdim \@tempskipa>\z@ \begingroup #6\relax
    \@hangfrom{\hskip #3\relax\@svsec}{\interlinepenalty\@M #8\par}\endgroup
    \ifnum#2>\@m \else \@tocwrite{#1}{#8}\fi
  \else
  \def\@svsechd{#6\hskip #3\@svsec
    \@ifnotempty{#8}{\ignorespaces#8\unskip
       \@addpunct.}\ifnum#2>\@m \else \@tocwrite{#1}{#8}\fi
  }\fi
  \global\@nobreaktrue
  \@xsect{#5}}
\makeatother

\usepackage{amsmath,amssymb,amsthm}
\usepackage{mathrsfs}
\usepackage{mathabx}\changenotsign
\usepackage{dsfont}

\usepackage[dvipsnames]{xcolor}
\usepackage{hyperref}
\hypersetup{
	colorlinks,
	linkcolor={red!60!black},
	citecolor={green!60!black},
	urlcolor={blue!60!black}
}

\definecolor{codelightgray}{gray}{0.8}
\definecolor{codeverylightgray}{gray}{0.9}

\usepackage{array,multirow,colortbl,enumerate}

\usepackage[open,openlevel=2,atend]{bookmark}

\usepackage[abbrev,msc-links,backrefs]{amsrefs}
\usepackage{amsrefs}
\usepackage{doi}

\renewcommand{\PrintDOI}[1]{\doi{#1}}

\usepackage[OT2, T1]{fontenc}
\usepackage{lmodern}
\usepackage[babel]{microtype}
\usepackage[english]{babel}

\DeclareRobustCommand{\rn}[1]{{\fontencoding{OT2}\selectfont#1}}

\usepackage{geometry}
\numberwithin{equation}{section}
\numberwithin{figure}{section}

\usepackage{enumitem}
\def\rmlabel{\upshape({\itshape \roman*\,})}

\let\polishlcross=\l
\def\l{\ifmmode\ell\else\polishlcross\fi}

\def\paragraph#1{	\noindent\textbf{#1.}\enspace}

\let\sm=\setminus

\makeatletter
\def\moverlay{\mathpalette\mov@rlay}
\def\mov@rlay#1#2{\leavevmode\vtop{   \baselineskip\z@skip \lineskiplimit-\maxdimen
		\ialign{\hfil$\m@th#1##$\hfil\cr#2\crcr}}}
\newcommand{\charfusion}[3][\mathord]{
	#1{\ifx#1\mathop\vphantom{#2}\fi
		\mathpalette\mov@rlay{#2\cr#3}
	}
	\ifx#1\mathop\expandafter\displaylimits\fi}
\makeatother

\newcommand{\dcup}{\charfusion[\mathbin]{\cup}{\cdot}}

\DeclareFontFamily{U}  {MnSymbolC}{}
\DeclareSymbolFont{MnSyC}         {U}  {MnSymbolC}{m}{n}
\DeclareFontShape{U}{MnSymbolC}{m}{n}{
	<-6>  MnSymbolC5
	<6-7>  MnSymbolC6
	<7-8>  MnSymbolC7
	<8-9>  MnSymbolC8
	<9-10> MnSymbolC9
	<10-12> MnSymbolC10
	<12->   MnSymbolC12}{}
\DeclareMathSymbol{\powerset}{\mathord}{MnSyC}{180}

\usepackage{tikz}
\usetikzlibrary{calc,decorations.pathmorphing,decorations.pathreplacing}
\usetikzlibrary{intersections}
\usetikzlibrary {arrows.meta} 
\pgfdeclarelayer{background}
\pgfdeclarelayer{foreground}
\pgfdeclarelayer{front}
\pgfsetlayers{background,main,foreground,front}

\usepackage{subcaption}
\let\epsilon=\varepsilon
\let\eps=\epsilon
\let\rho=\varrho
\let\theta=\vartheta

\def\FF{{\mathds F}}
\def\NN{{\mathds N}}

\def\ZZ{{\mathds Z}}

\def\RR{{\mathds R}}

\theoremstyle{plain}
\newtheorem{thm}{Theorem}[section]
\newtheorem{theorem}[thm]{Theorem}

\newtheorem{prop}[thm]{Proposition}
\newtheorem{proposition}[thm]{Proposition}
\newtheorem{claim}[thm]{Claim}
\newtheorem{fact}[thm]{Fact}
\newtheorem{cor}[thm]{Corollary}
\newtheorem{lemma}[thm]{Lemma}

\theoremstyle{definition}

\newtheorem{dfn}[thm]{Definition}
\newtheorem{definition}[thm]{Definition}

\usepackage{accents}

\let\lra=\longrightarrow
\let\phi=\varphi

\DeclareSymbolFont{stmry}{U}{stmry}{m}{n}
\DeclareMathSymbol\arrownot\mathrel{stmry}{"58}
\DeclareMathSymbol\Arrownot\mathrel{stmry}{"59}

\def\Sym{\mathrm{Sym}}
\def\sfr{\mathrm{sf}}
\def\SFR{\mathrm{SF}}
\def\SFRR{\widetilde{\mathrm{SF}}}
\let\vn=\varnothing

\def\VL{\mathrm{VL}}

\usepackage{nicefrac, xfrac}

\usepackage[makeroom]{cancel}
\usepackage{tikz}
\usetikzlibrary{shapes.misc}

\tikzset{cross/.style={cross out, draw=black, minimum size=2*(#1-\pgflinewidth), inner sep=0pt, outer sep=0pt},
cross/.default={6pt}}

\tikzset{crossnor/.style={cross out, draw=black, minimum size=2*(#1-\pgflinewidth), inner sep=0pt, outer sep=0pt},
crossnor/.default={5pt}}

\tikzset{crossleila/.style={cross out, draw=violet, minimum size=2*(#1-\pgflinewidth), inner sep=0pt, outer sep=0pt},
crossleila/.default={6pt}}

\tikzset{crosstr/.style={cross out, draw=teal, minimum size=2*(#1-\pgflinewidth), inner sep=0pt, outer sep=0pt},
crosstr/.default={6pt}}

\tikzset{crossgr/.style={cross out, draw=black, minimum size=2*(#1-\pgflinewidth), inner sep=0pt, outer sep=0pt},
crossgr/.default={10pt}}

\tikzset{crossgrtr/.style={cross out, draw=teal, minimum size=2*(#1-\pgflinewidth), inner sep=0pt, outer sep=0pt},
crossgrtr/.default={10pt}}

\tikzset{crosspt/.style={cross out, draw=black, minimum size=2*(#1-\pgflinewidth), inner sep=0pt, outer sep=0pt},
crosspt/.default={3pt}}

\tikzset{crosspttr/.style={cross out, draw=teal, minimum size=2*(#1-\pgflinewidth), inner sep=0pt, outer sep=0pt},
crosspttr/.default={4.2pt}}

\let\sm=\smallsetminus

\begin{document}
\title[Large sum-free sets in finite vector spaces II]{Large sum-free sets in finite vector spaces II.}

\author[Christian Reiher]{Christian Reiher}

\address{Fachbereich Mathematik, Universit\"at Hamburg, Hamburg, Germany}
\email{christian.reiher@uni-hamburg.de}

\author[Sofia Zotova]{Sofia Zotova}
\address{Mathematisches Institut, Universit\"at Bonn, Bonn, Germany}
\email{s87szoto@uni-bonn.de}

\begin{abstract}
	Answering a question of Leo Versteegen, we prove that for \(n\ge 3\) every 
	sum-free set \(A\subseteq\FF_5^n\) with \(|A|\ge 28\cdot 5^{n-3}\) is either 
	contained in the union of two parallel hyperplanes, or isomorphic 
	to \(\Lambda\times \FF_5^{n-3}\), where \(\Lambda\subseteq \FF_5^3\) denotes 
	a certain sum-free set of size~\(28\) discovered by Vsevolod Lev and Leo Versteegen.  
\end{abstract}

\maketitle

\section{Introduction}
A subset \(A\) of an abelian group \(G\) is called {\it sum-free} if there are no 
solutions of \(x+y=z\) with \(x, y, z\in A\). The study of such sets has a long history, 
which can be traced back to work of Issai Schur~\cite{Schur} motivated by Fermat's last 
theorem. 

Given a finite abelian group \(G\) we write \(\sfr_0(G)\) for the maximum cardinality of 
a sum-free subset of \(G\). Building on earlier work by many researchers, the determination 
of this group invariant was completed by Ben Green and Imre Ruzsa~\cite{Gr05}. 
Once~\(\sfr_0(G)\) is known, one can ask further for a classification 
of the collection 
\[
	\SFRR_0(G)=\{A\subseteq G\colon \text{\(A\) is sum-free and } |A|=\sfr_0(G)\}
\]
of all maximum sum-free sets. The solution of this problem was completed by Balasubramanian, 
Prakash, and Ramana~\cite{Ba16}. The next natural question is to determine the largest 
possible size~\(\sfr_1(G)\) of a sum-free set \(A\subseteq G\) that it not contained in 
any set from \(\SFRR_0(G)\). 

In an early contribution, Davydov and Tombak~\cite{DT} solved this problem for binary vector 
spaces, where it has connections to coding theory.
We have \(\sfr_0(\FF_2^n)=2^{n-1}\) for every \(n\ge 1\) and the only maximum sum-free 
sets are affine hyperplanes not containing the origin. For~\({n\le 3}\) all sum-free subsets
can be covered by such hyperplanes, while for \(n\ge 4\) Davydov and Tombak 
proved \(\sfr_1(\FF_2^n)=5\cdot 2^{n-4}\). 

For ternary vector spaces, the analogous problem was solved by Vsevolod Lev~\cite{VL05}.
The maximum sum-free subsets of~\(\FF_3^n\) are again affine hyperplanes not passing through 
the origin, so they have size~\(3^{n-1}\). In the nontrivial case \(n\ge 3\) the main result 
of~\cite{VL05} states~\({\sfr_1(\FF_3^n)=5\cdot 3^{n-3}}\).

Continuing this line of research, Vsevolod Lev~\cite{VL23}, followed by Leo 
Versteegen~\cite{LV23}, initiated the investigation of \(\sfr_1(\FF_5^n)\). 
Perhaps interestingly, for all primes \(p\ge 7\) and all dimensions~\(n\ge 1\) 
the invariant \(\sfr_1(\FF_p^n)\) has been determined in the meantime 
(see~\cite{RZ24a} and the references therein), so \(p=5\) is the last open case. 
It is well known that \(\sfr_0(\FF_5^n)=2\cdot 5^{n-1}\) and that the only 
maximum sum-free subsets of \(\FF_5^n\) are sets of the form \(H\cup (-H)\), where \(H\) 
denotes an affine hyperplane not containing the origin (cf.\ Fact~\ref{f:23}). 
A sum-free set \(A\subseteq \FF_5^n\) 
is called {\it normal} if it is a subset of such a set. Thus \(\sfr_1(\FF_5^n)\) is the largest
cardinality of a non-normal sum-free subset of \(\FF_5^n\). Since all sum-free subsets 
of \(\FF_5\) are normal, this invariant is only interesting for~\({n\ge 2}\). 
Leo Versteegen~\cite{LV23} showed \(\sfr_1(\FF^2_5)=5\) and 
based on his analysis it is easy to classify all extremal examples.   

\begin{theorem}\label{5n=2}
	If a sum-free set \(A\subseteq \FF_5^2\) with \(|A|\ge 5\) is not normal, then it is 
	isomorphic to one of the two sets displayed in Figure~\ref{fig:5sets}.        
\end{theorem}

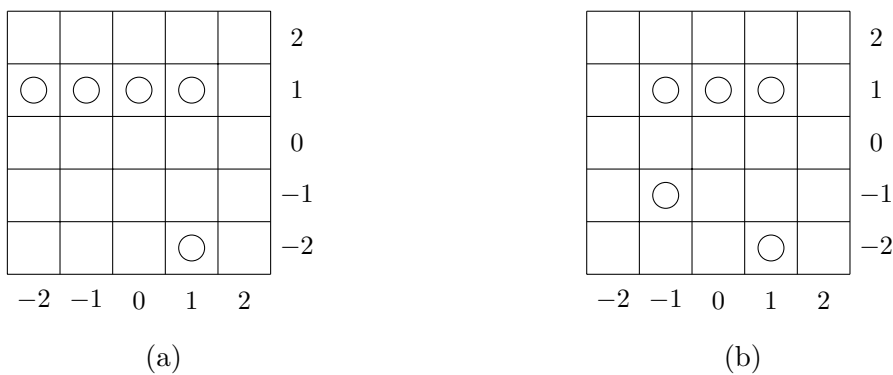
\begin{figure}[h]
	\centering
\hspace{4em}
	\begin{subfigure}[b]{.32\textwidth}
	\centering
		\begin{tikzpicture}[scale=0.7]
		\foreach \x in {-2,...,3} \draw (-2.5, \x-.5)--(2.5, \x-.5) (\x-.5, -2.5)--(\x-.5, 2.5);
		\foreach \x in {-2,...,2} 
			\node at (\x,0) [label={[yshift=-2.5cm]\footnotesize\(\x\)}] {};
		\foreach \y in {-2,...,2} 
			\node at (0,\y) [label={[xshift=2.1cm, yshift=-0.4cm]\footnotesize\(\y\)}] {};
		\foreach \x/\y in {-2/1,-1/1,0/1,1/1,1/-2} \draw (\x,\y) circle (7pt);
		\end{tikzpicture}
		\caption{} 
	\end{subfigure}
\hfill
	\begin{subfigure}[b]{.32\textwidth}
	\centering
		\begin{tikzpicture}[scale=0.7]
		\foreach \x in {-2,...,3} \draw (-2.5, \x-.5)--(2.5, \x-.5) (\x-.5, -2.5)--(\x-.5, 2.5);
		\foreach \x in {-2,...,2} 
			\node at (\x,0) [label={[yshift=-2.5cm]\footnotesize\(\x\)}] {};
		\foreach \y in {-2,...,2} 
			\node at (0,\y) [label={[xshift=2.1cm, yshift=-0.4cm]\footnotesize\(\y\)}] {};		
		\foreach \x/\y in {-1/1,0/1,1/1,1/-2,-1/-1} \draw (\x,\y) circle (7pt);
		\end{tikzpicture}
	\caption{} 
	\end{subfigure}
\hspace{4em}
	\caption{Non-normal sum-free sets of size \(5\)}		
	\label{fig:5sets}
\end{figure}

Here and throughout this work~\(\FF_5^2\) is drawn as a \((5\times5)\)-grid 
whose~\(25\) {\it boxes} represent the points of~\(\FF_5^2\). For the sake 
of symmetry, we let the coordinates in \(\FF_5\) run from \(-2\) to \(2\), 
so that the origin \((0,0)\) corresponds to the box in the centre. The five 
elements of our non-normal sum-free sets are indicated by circles. 
Moreover, two subsets of \(\FF_5^2\) are said to be {\it isomorphic} if they 
agree up to a change of coordinates, i.e., if they are in the same orbit of 
the natural action of \(\mathrm{Aut}(\FF_5^2)\) on the subsets of~\(\FF_5^2\). 

For \(n\ge 3\) the upper bounds \(\sfr_1(\FF_5^n)<1.5\cdot 5^{n-1}\) and then 
\(\sfr_1(\FF_5^n)\le 6\cdot 5^{n-2}\) were obtained in~\cite{VL23}*{Theorem~3}
and \cite{LV23}*{Theorem~1.4}, respectively. As to lower bounds, a correspondence  
between Vsevolod Lev and Leo Versteegen led the latter to the following 
configuration.    

\begin{dfn}\label{dfn:VL}
	Let \(\Lambda\subseteq \FF_5^3\) be the set 
	\begin{align*}
		\Lambda&=\{(1,0),(-1,0),(0,1),(0,-1)\}\times\FF_5 \dcup \{(-1,2,0),(-1,2,1), \\ 
			& \hspace{4em} (2,1,1), (2,1,2),(1,-2,0),(1,-2,-1),(-2,-1,-1),(-2,-1,-2)\}\,.
	\end{align*}
	For \(n\ge 3\) a subset of \(\FF_5^n\) isomorphic to \(\Lambda\times \FF_5^{n-3}\) 
	is called a {\it \(\VL\)-set}.
\end{dfn}

Our choice of terminology reflects the coincidental fact that, up to symmetry, Vsevolod Lev 
and Leo Versteegen have the same initials. It is not difficult to verify that \(\Lambda\) is 
sum-free, but not normal. The same holds, more generally, for all \(\VL\)-sets, which proves 
the lower bound~\({\sfr_1(\FF_5^{n-3})\ge 28\cdot 5^{n-3}}\). Our main result asserts that 
this construction is optimal. 

\begin{theorem}\label{thm:main}
	For every \(n\ge 3\) we have \(\sfr_1(\FF_5^n)=28\cdot5^{n-3}\). Moreover, 
	every sum-free set~\({A\subseteq \FF_5^{n-3}}\) of size \(|A|\ge 28\cdot 5^{n-3}\) 
	is either normal or a \(\VL\)-set.
\end{theorem} 

\section{Preliminaries}

\subsection{Kneser's theorem}
Given two subsets \(A\) and \(B\) of an abelian group \(G\) we write
\[
	A+B=\{a+b\colon a\in A \text{ and }  b\in B\}
\]
for their {\it sum set}. So \(A\) is sum-free if and only if \((A+A)\cap A=\vn\). 
We need a general lower bound on the cardinalities of such sum sets due to 
Kneser~\cites{Kn53, Kn55}. Its formulation involves the concept of the 
{\it symmetry set} of an arbitrary subset \(X\subseteq G\), which is defined 
by  
\[
	\Sym(X)=\{g\in G \colon g+X=X\}\,.
\]
Clearly, this set is a subgroup of \(G\) and \(X\) is a union of cosets of \(\Sym(X)\). 

\begin{theorem}[Kneser]\label{Kneser}
    If \(A\) and \(B\) denote two finite nonempty subsets of an abelian group~\(G\) 
    and \(K=\Sym(A+B)\), then 
    \[
        |A+B|\ge |A+K|+|B+K|-|K|\,.
    \]
\end{theorem}

The right side is clearly a multiple of \(|K|\) and, if \(G\) is finite, so is \(|G|\). 
Therefore, Kneser's theorem immediately implies the following result, which could also 
be proved in a much more direct manner. 

\begin{cor}\label{Bdumm}
    If two subsets \(A\) and \(B\) of a finite abelian group \(G\) satisfy \(|A|+|B|>|G|\), 
    then \(A+B=G\). 
\end{cor}

Later on, we shall encounter more substantial applications of Kneser's theorem. To illustrate 
the general idea, we quickly derive the well-known formula for \(\sfr_0(\FF_5^n)\) from this 
result. 

\begin{fact}\label{f:23}
	If \(A, B, C\subseteq \FF_5^n\) are three nonempty subsets satisfying \((A+B)\cap C=\vn\), 
	then 
	\[
		|A|+|B|+|C|\le 6\cdot 5^{n-1}\,.
	\]
	In particular, we have \(\sfr_0(\FF_5^n)=2\cdot 5^{n-1}\) for every natural number~\(n\). 
\end{fact}

\begin{proof}
	Consider the symmetry set \(K=\Sym(A+B)\). This is a linear subspace of \(\FF_5^n\) 
	and~\({A+B}\) is a union of translates of \(K\). Since \(A+B\) is disjoint to the nonempty 
	set \(C\), the dimension of \(K\) can be at most \(n-1\), whence \(|K|\le 5^{n-1}\). 
	Moreover, Kneser's theorem yields 
	\[
		5^n
		\ge 
		|A+B|+|C|
		\ge 
		|A+K|+|B+K|-|K|+|C|
		\ge 
		|A|+|B|+|C|-5^{n-1}
	\]
	and the desired upper bound \(|A|+|B|+|C|\le 6\cdot 5^{n-1}\) follows. 
	
	In the special case \(A=B=C\) this tells us \(|A|\le 2\cdot 5^{n-1}\) 
	for every sum-free subset \(A\subseteq \FF_5^n\), which proves the 
	upper bound \(\sfr_0(\FF_5^n)\le 2\cdot 5^{n-1}\). In the other direction,  
	for any affine hyperplane \(H\) not passing through the origin the sum-free set 
	\(H\cup(-H)\) exemplifies the lower bound \(\sfr_0(\FF_5^n)\ge 2\cdot 5^{n-1}\). 
\end{proof}

\subsection{Proof of Theorem~\ref{5n=2}}\label{n=2}

In this subsection, we classify the maximum non-normal sum-free subsets of \(\FF^2_5\). 
The main step towards this goal is the following lemma of Leo 
Versteegen~\cite{VL23}*{Lemma 3.1}, which we quote in a coordinate-free way. 

\begin{lemma}\label{3collin}
    If \(A\subseteq\FF_5^2\) is sum-free and \(|A|\ge 5\), then at least three elements 
    of \(A\) are contained in a line not passing through the origin. \qed
\end{lemma}

\begin{proof}[Proof of Theorem~\ref{5n=2}]
	Let \(A\subseteq\FF_5^2\) be a non-normal sum-free set of size \(|A|\ge 5\). 
	By Lemma~\ref{3collin}, we can suppose without loss of generality that 
	\((-1,1),(0,1),(1,1)\in A\). These three points are drawn as circles in 
	Figure~\ref{fig:n=2.1}. All their sums, differences, and halves are indicated 
	by crosses, because they cannot belong to the sum-free set \(A\) anymore. 
	
\begin{figure}[h]
	\centering
\hspace{4em}
	\begin{subfigure}[b]{.32\textwidth}
	\centering
		\begin{tikzpicture}[scale=0.7]
		\foreach \x in {-2,...,3} \draw (-2.5, \x-.5)--(2.5, \x-.5) (\x-.5, -2.5)--(\x-.5, 2.5);
		\foreach \x in {-2,...,2} 
			\node at (\x,0) [label={[yshift=-2.5cm]\footnotesize\(\x\)}] {};
		\foreach \y in {-2,...,2} 
			\node at (0,\y) [label={[xshift=2.1cm, yshift=-0.4cm]\footnotesize\(\y\)}] {};
		\foreach \x in {-1,0,1} \draw (\x,1) circle (7pt);
		\foreach \x in {-2,...,2} \foreach \y in {0,2} \draw (\x,\y) node[crossnor]{};
		\foreach \x in {-2,0,2} \draw (\x,-2) node[crossnor]{};
		\end{tikzpicture}
		\caption{} \label{fig:n=2.1}
	\end{subfigure}
\hfill
	\begin{subfigure}[b]{.32\textwidth}
	\centering
		\begin{tikzpicture}[scale=0.7]
		\foreach \x in {-2,...,3} \draw (-2.5, \x-.5)--(2.5, \x-.5) (\x-.5, -2.5)--(\x-.5, 2.5);
		\foreach \x in {-2,...,2} 
			\node at (\x,0) [label={[yshift=-2.5cm]\footnotesize\(\x\)}] {};
		\foreach \y in {-2,...,2} 
			\node at (0,\y) [label={[xshift=2.1cm, yshift=-0.4cm]\footnotesize\(\y\)}] {};		
		\foreach \x/\y in {-1/1,0/1,1/1,1/-2} \draw (\x,\y) circle (7pt);
		\foreach \x in {-2,...,2} \foreach \y in {0,2} \draw (\x,\y) node[crossnor]{};
		
		\foreach \x/\y in {2/1,-2/-1,0/-1,1/-1,2/-1,-2/-2,-1/-2,0/-2,2/-2}
			\draw (\x,\y) node[crossnor]{};
		\end{tikzpicture}
	\caption{} \label{fig:n=2.2}
	\end{subfigure}
\hspace{4em}
	\caption{}\label{fig:1147}
\end{figure}

	Since the set \(A\) fails to be normal, it cannot be covered by \(\FF_5\times \{-1, 1\}\)
	and, therefore, one of \((\pm 1, -2)\) needs to be in \(A\). For reasons of symmetry 
	we can suppose \((1, -2)\in A\). As shown in Figure~\ref{fig:n=2.2}, this excludes several 
	further points from~\(A\). 
	
	Altogether, the only points whose status is currently unknown are \((-2,1)\) and \((-1,-1)\).
	Each of them could be in \(A\), but due to \((-1,-1)-(-2,1)=(1,-2)\) they cannot be in \(A\)
	at the same time. This proves \(|A|=5\) and that \(A\) is isomorphic to one of the two sets 
	in Figure~\ref{fig:5sets}. Both of these configurations are indeed sum-free.
\end{proof}

\subsection{Overview}\label{sec:overview}

As the proof of Theorem~\ref{thm:main} is fairly long, we want to break it into three 
major steps, that are mostly independent of each other. The central idea is that given 
a dense sum-free set \(A\subseteq \FF_5^n\) and a linear 
epimorphism \(\phi\colon \FF_5^n\lra \FF_5^m\) onto a low-dimensional space, the associated 
function 
\[
	f_\phi^A\colon \FF_5^m\lra \RR_{\ge 0}\,, 
	\qquad x\longmapsto \frac{|A\cap \phi^{-1}(x)|}{5^{n-3}}
\]
has non-trivial properties, which can then be used to deduce structural information on \(A\). 
For instance, we shall see later that if \(m=1\), \(|A|\ge 28\cdot 5^{n-3}\), and the 
support of \(f^A_\phi\) has size at most three, then \(A\) has to be normal 
(see Lemma~\ref{lem:3planes}). When \(m=2\) some properties such functions \(f^A_\phi\) need 
to have are gathered in the following concept (cf.\ Lemma~\ref{lem:1922}). 

\begin{dfn}\label{dfn:fish}
	A function \(f\colon \FF_5^2\lra \RR_{\ge 0}\) is said to be {\it fishy} if 
	\begin{enumerate}[label=(\(F\arabic*\))]
		\item\label{it:fish1} \(\|f\|_1\ge 28\), \(\|f\|_\infty\le 5\), 
		\item\label{it:fish2} there is no \(X\subseteq \FF_5^2\) such 
			that \(f(X)=\sum_{P\in X} f(P)\) is half of an odd integer, 
		\item\label{it:fish3} and for all \(P, Q\in \FF_5^2\) and all \(\eps\in\{-1, 1\}\) 
			with \(f(P), f(Q), f(P+\eps Q)>0\) 
			we have \[f(P)+f(Q)+\lceil f(P+\eps Q)\rceil\le 6\,.\]
	\end{enumerate}
\end{dfn}

Here the \(1\)-norm and the infinity norm of \(f\) are defined 
by \(\|f\|_1=\sum_{P\in\FF_5^2} |f(P)|\) 
and \(\|f\|_\infty=\max\bigl\{|f(P)|\colon P\in \FF_5^2\bigr\}\),
respectively; since \(f\ge 0\), the absolute value signs could be omitted. Three examples of 
fishy functions relevant to our classification of the extremal configurations 
(i.e., the \(\VL\)-sets) are shown in Figure~\ref{fig:prop-fishy}. Empty boxes in these figures 
indicate that the corresponding function attains the value~\(0\). 

\begin{figure}[h]
	\centering
	\begin{subfigure}[b]{.32\textwidth}
	\centering
		\begin{tikzpicture}[scale=0.7]
		\foreach \x in {-2,...,3} \draw (-2.5, \x-.5)--(2.5, \x-.5) (\x-.5, -2.5)--(\x-.5, 2.5);
		\foreach \x in {-2,...,2} 
			\node at (\x,0) [label={[yshift=-2.5cm]\footnotesize\(\x\)}] {};
		\foreach \y in {-2,...,2} 
			\node at (0,\y) [label={[xshift=2.1cm, yshift=-0.4cm]\footnotesize\(\y\)}] {};
		\foreach \x/\y in {1/0,-1/0,0/1,0/-1} \draw (\x,\y) node{\(5\)};
		\foreach \x/\y in {-1/2,2/1,1/-2,-2/-1} \draw (\x,\y) node{\(2\)};
		\end{tikzpicture}
	\caption{\(f_\alpha\)}\label{fig:fA}
	\end{subfigure}
\hfill
	\begin{subfigure}[b]{.32\textwidth}
	\centering
		\begin{tikzpicture}[scale=.7]
		\foreach \x in {-2,...,3} \draw (-2.5, \x-.5)--(2.5, \x-.5) (\x-.5, -2.5)--(\x-.5, 2.5);
		\foreach \x in {-2,...,2} 
			\node at (\x,0) [label={[yshift=-2.5cm]\footnotesize\(\x\)}] {};
		\foreach \y in {-2,...,2} 
			\node at (0,\y) [label={[xshift=2.1cm, yshift=-0.4cm]\footnotesize\(\y\)}] {};
		\foreach \x in {-2,...,2} \foreach \y in {-2,-1,1,2} \draw (\x,\y) node{\(1\)};
		\foreach \x in {-1,1} \draw (\x,0) node{\(4\)};
		\end{tikzpicture}
	\caption{\(f_\beta\)}\label{fig:fB}
	\end{subfigure}
\hfill
	\begin{subfigure}[b]{.32\textwidth}
	\centering
		\begin{tikzpicture}[scale=.7]
		\foreach \x in {-2,...,3} \draw (-2.5, \x-.5)--(2.5, \x-.5) (\x-.5, -2.5)--(\x-.5, 2.5);
		\foreach \x in {-2,...,2} 
			\node at (\x,0) [label={[yshift=-2.5cm]\footnotesize\(\x\)}] {};
		\foreach \y in {-2,...,2} 
			\node at (0,\y) [label={[xshift=2.1cm, yshift=-0.4cm]\footnotesize\(\y\)}] {};
		\foreach \x in {-2,-1,1,2} \foreach \y in {-2,2} \draw (\x,\y) node{\(1\)};
		\foreach \x in {-2,2} \foreach \y in {-1,1} \draw (\x,\y) node{\(1\)};
		\foreach \x/\y in {1/0,-1/0,0/1,0/-1} \draw (\x,\y) node{\(4\)};
		\end{tikzpicture}
	\caption{\(f_\gamma\)}\label{fig:fC}
	\end{subfigure}
	\caption{Three functions}\label{fig:prop-fishy}
\end{figure}
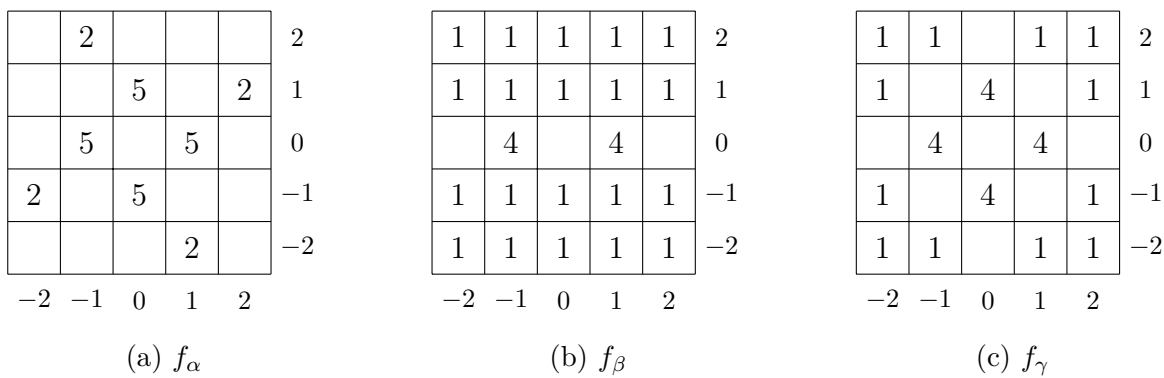

The first major step towards Theorem~\ref{thm:main} is the following result on 
fishy functions.

\begin{prop}\label{prop:fish}
	If \(f\colon \FF_5^2\lra \RR_{\ge 0}\) is fishy and \(\|f\|_\infty>3\), 
	then either the support of \(f\) can be covered by three parallel lines 
	or \(f\) is, modulo some automorphism of \(\FF_5^2\), one of the three 
	functions depicted in Figure~\ref{fig:prop-fishy}. More explicitly, this 
	means that there are an automorphism~\(\phi\) of~\(\FF_5^2\) and a 
	letter \(\tau\in\{\alpha, \beta, \gamma\}\) such that \(f=f_\tau\circ\phi\).  
\end{prop}

By means of Kneser's theorem, this can be shown to have the following consequence 
on sum-free sets. 

\begin{proposition}\label{prop:0133}
	If for some \(n\ge 3\) a sum-free set \(A\subseteq \FF^n_5\) has 
	size \(|A|\ge 28\cdot 5^{n-3}\) and some affine subspace \(T\)
	of \(\FF^n_5\) with codimension \(2\) satisfies \(|A\cap T|>3\cdot 5^{n-3}\),
	then \(A\) is normal or a~\(\VL\)-set. 
\end{proposition}

We also need to address certain functions \(f\colon \FF_5^3\lra \RR_{\ge 0}\).
Prior to stating the relevant result, we introduce further notation and terminology. 
For every point \(P=(x, y)\in \FF_5^2\) we write 
\[
	L_P=L_{x, y}=\bigl\{(x, y, z)\colon z\in \FF_5\bigr\}
\] 
for its preimage with respect to the projection \(\FF_5^3\lra \FF_5^2\) to the 
first two coordinates. This is a line in \(\FF_5^3\) parallel to the third coordinate 
axis. Given any function \(f\colon \FF_5^3\lra \RR\) we call the function 
\(g\colon \FF_5^2\lra\RR\) defined by \(g(P)=f(L_{P})\) the {\it standard projection} 
of \(f\).    

\begin{dfn}
	Consider a function \(f\colon \FF_5^3\lra \RR_{\ge 0}\) with support \(I\) 
	and standard projection~\(g\) such that
	\begin{enumerate}[label=(\(A\arabic*\))]
		\item\label{it:owl1} \(\|f\|_\infty\le 1\),
		\item\label{it:owl2} \(g\) is fishy,
		\item\label{it:owl3} and if \(\eps\in\{-1, 1\}\), \(P, Q\in\FF_5^3\), 
			and \(P+\eps Q\in I\), then \(f(P)+f(Q)\le 1\). 
	\end{enumerate}
	We call \(f\) {\it acceptable} if in addition there exists a special 
	point \(P_\star\in \FF_5^2\) such that 
	\begin{enumerate}[label=(\(A\arabic*\)), resume]
		\item\label{it:owl4} \(g(P_\star)>2.8\) and \(g(-P_\star)>2.5\), 
		\item\label{it:owl5} \(|L_{P_\star}\cap I|=3\), 
		\item\label{it:owl6} and for all \(Q\in \FF_5^2\) with \(g(Q), g(P_\star+Q)>0\) 
			and \(g(Q)+g(P_\star+Q)>2.2\) we have 
			\[|L_Q\cap I|+|L_{P_\star+Q}\cap I|=3\,.\]
	\end{enumerate}
\end{dfn}

Here is our main result on acceptable functions. 

\begin{proposition}\label{prop:klein}
	For every acceptable function \(f\colon \FF_5^3\lra \RR_{\ge 0}\) there exists 
	a line \(L\subseteq \FF_5^3\) such that \(f(L)>3\).
\end{proposition}

Combined with Proposition~\ref{prop:0133} and another application of Kneser's theorem 
this will eventually lead us to the following statement. 

\begin{proposition}\label{prop:0135}
	Let \(A\subseteq \FF_5^n\) be a sum-free set of size \(|A|\ge 28\cdot 5^{n-3}\),
	where \(n\ge 3\). If there exists an affine subspace \(T\) of \(\FF^n_5\) with 
	codimension \(2\) such that \(|A\cap T|+|A\cap (-T)|>28\cdot 5^{n-4}\), then \(A\) 
	is normal or a \(\VL\)-set.
\end{proposition}

We conclude this section by explaining why this result implies our main theorem. 

\begin{proof}[Proof of Theorem~\ref{thm:main} assuming Proposition~\ref{prop:0135}]
     We will establish the following statement by induction on \(n\ge 2\). 
     
     \begin{quotation}
     	\it 
		Every sum-free set \(A\subseteq\FF_5^n\) of size \(|A|\geq28\cdot5^{n-3}\) 
		is normal or \(\VL\).
	\end{quotation}

    The base case, $n=2$, amounts to the fact that every sum-free set \(A\subseteq\FF_5^2\)
    with at least six elements is normal, which is a simple consequence of Theorem~\ref{5n=2}. 
    In the inductive step from~\({n-1}\) to \(n\) we consider any sum-free set 
    \(A\subseteq\FF_5^n\) of size~\({|A|\geq28\cdot5^{n-3}}\), where \(n\ge 3\). 
    Notice that every two-dimensional subspace \(E\le \FF^n_5\) contains~\(24\) 
    nonzero vectors, while each nonzero \(v\in \FF^n_5\) is contained in the same number 
    of two-dimensional subspaces~\(E\). So because of \(0\not\in A\) an averaging  
    argument yields a two-dimensional subspace~\({E\le \FF^n_5}\) such that 
	\[
        |A\cap E|
        \ge
        \left\lceil\frac{24|A|}{5^n-1}\right\rceil
        \ge
        \left\lceil\frac{24\cdot 28}{125}\right\rceil
        =
        6\,.
    \]
    
    By Theorem~\ref{5n=2} the sum-free subset \(A\cap E\) of \(E\) is normal and, hence, there 
    is a one-dimensional space \(L\le E\) intersecting \(A\) in a set of the 
    form \(A\cap L=\{v, -v\}\). We now want to find a hyperplane \(H\) such 
    that \(L\le H\le \FF^n_5\) and \(|A\cap H|\) is large. 
    Since all vectors in~\({\FF^n_5\sm L}\) are in the same number 
    of such hyperplanes, another averaging argument provides 
    a hyperplane \(H\) satisfying 
    \begin{align*}
        |A\cap H|
        &\ge
        2+\left\lceil\frac{(|A|-2)(5^{n-1}-5)}{5^n-5}\right\rceil
        \ge
        2+\frac{(28\cdot5^{n-3}-2)\cdot(5^{n-1}-5)}{5^n} \\
        &>
        2+28\cdot5^{n-4}-\frac{38}{25}
        >
        28\cdot5^{n-4}\,.
    \end{align*}
    
    Since the last inequality is strict, the induction hypothesis applied in \(H\) 
    shows that~\({A\cap H}\) is a normal subset of~\(H\).
    This means that there exists an affine hyperplane~\(T\) in~\(H\) such 
    that~\({A\cap H\subseteq T\cup (-T)}\). Now~\(T\) satisfies the hypothesis 
    of Proposition~\ref{prop:0135} and, consequently,~\(A\) is normal or 
    a \(\VL\)-set. 
\end{proof}

\subsection*{Organisation}
We will prove Proposition~\ref{prop:fish} in Section~\ref{sec:fish} and 
Proposition~\ref{prop:klein} in Section~\ref{sec:klein}. The derivation 
of Propositions~\ref{prop:0133} and~\ref{prop:0135} will then be carried 
out in Section~\ref{sec:kneser}. We conclude with some open problems in 
Section~\ref{sec:conclude}.
  
\section{Fishy functions}\label{sec:fish}

Throughout this section, which is dedicated to the proof of Proposition~\ref{prop:fish}, 
we denote for every \(j\in\FF_5\) the \(j^{\mathrm{th}}\) row of \(\FF_5^2\), i.e., 
the set \(\FF_5\times\{j\}\) by \(R_j\). 
For reasons that will become apparent later, we require the following observation. 

\begin{lemma}\label{lem:Spiel}
    If \(I\subseteq \FF_5^2\) has size at least \(9\) and the complement of \(I-I\)
    contains two linearly independent vectors, then \(I\) is contained in the union 
    of three parallel lines.
\end{lemma}
    
\begin{figure}[h]
	\centering
	\begin{subfigure}[b]{.32\textwidth}
	\centering
		\begin{tikzpicture}[scale=0.7]
		\foreach \x in {-2,...,3} \draw (-2.5, \x-.5)--(2.5, \x-.5) (\x-.5, -2.5)--(\x-.5, 2.5);
		\foreach \x in {-2,...,2} 
			\node at (\x,0) [label={[yshift=-2.5cm]\footnotesize\(\x\)}] {};
		\foreach \y in {-2,...,2} 
			\node at (0,\y) [label={[xshift=2.1cm, yshift=-0.4cm]\footnotesize\(\y\)}] {};
		\foreach \x in {-1, 1} \draw (\x, 0) circle(7pt);
		\foreach \x/\y in {-1/1, 0/0, 1/-1, -2/0, 1/1, 2/0, -1/-1}
			\draw (\x,\y) node[cross]{};
		\end{tikzpicture}
		\caption{}\label{fig:Spiel1}
	\end{subfigure}
\hfill
	\begin{subfigure}[b]{.32\textwidth}
	\centering
		\begin{tikzpicture}[scale=0.7]
		\foreach \x in {-2,...,3} \draw (-2.5, \x-.5)--(2.5, \x-.5) (\x-.5, -2.5)--(\x-.5, 2.5);
		\foreach \x in {-2,...,2} 
			\node at (\x,0) [label={[yshift=-2.5cm]\footnotesize\(\x\)}] {};
		\foreach \y in {-2,...,2} 
			\node at (0,\y) [label={[xshift=2.1cm, yshift=-0.4cm]\footnotesize\(\y\)}] {};
		\foreach \x/\y in {1/0, -1/0, 0/1, 0/-1} \draw (\x,\y) circle(7pt);
		\foreach \x/\y in {-1/1, 0/0, 1/-1, -2/0, 1/1, 2/0, -1/-1, 0/2, 0/-2}
			\draw (\x,\y) node[cross]{};
		\end{tikzpicture}
		\caption{}\label{fig:Spiel2}
	\end{subfigure}
\hfill
	\begin{subfigure}[b]{.32\textwidth}
	\centering
		\begin{tikzpicture}[scale=0.7]
		\foreach \x in {-2,...,3} \draw (-2.5, \x-.5)--(2.5, \x-.5) (\x-.5, -2.5)--(\x-.5, 2.5);
		\foreach \x in {-2,...,2} 
			\node at (\x,0) [label={[yshift=-2.5cm]\footnotesize\(\x\)}] {};
		\foreach \y in {-2,...,2} 
			\node at (0,\y) [label={[xshift=2.1cm, yshift=-0.4cm]\footnotesize\(\y\)}] {};
		\foreach \x/\y in {1/0, -1/0, 0/1, 0/-1, 2/1} \draw (\x,\y) circle(7pt);
		\foreach \x/\y in {-1/1, 0/0, 1/-1, -2/0, 1/1, 2/0, -1/-1, 0/2, 0/-2, -2/1, 2/2}
			\draw (\x,\y) node[cross]{};
		\end{tikzpicture}
		\caption{}\label{fig:Spiel3}
	\end{subfigure}
	\caption{}
	\vspace{-1em}
\end{figure}

\begin{proof}
    As the statement is invariant under affine maps, we can assume that \((0, 1)\) 
    and~\((1, 0)\) are not in \(I-I\). This means that no two elements of \(I\) are 
    horizontally or vertically adjacent. Therefore, each row contains at most two 
    members of \(I\). By the box principle and translation invariance, we can suppose 
    that with the possible exception of \(R_{-2}\) all rows contain exactly two elements 
    of \(I\) and that \((0, \pm 1)\) are in \(R_0\cap I\). This situation is depicted 
    in Figure~\ref{fig:Spiel1}, where members of~\(I\) are represented as circles and 
    the crosses indicate boxes that are certainly not in~\(I\)---because they are 
    neighbours of boxes in~\(I\).   
            
\begin{figure}[h]
	\centering
	\begin{subfigure}[b]{.32\textwidth}
	\centering
		\begin{tikzpicture}[scale=0.7]
		\foreach \x in {-2,...,3} \draw (-2.5, \x-.5)--(2.5, \x-.5) (\x-.5, -2.5)--(\x-.5, 2.5);
		\foreach \x in {-2,...,2} 
			\node at (\x,0) [label={[yshift=-2.5cm]\footnotesize\(\x\)}] {};
		\foreach \y in {-2,...,2} 
			\node at (0,\y) [label={[xshift=2.1cm, yshift=-0.4cm]\footnotesize\(\y\)}] {};
		\foreach \x/\y in {1/0, -1/0, 0/1, 0/-1, 2/1, 1/2} \draw (\x,\y) circle(7pt);
		\foreach \x/\y in {-1/1, 0/0, 1/-1, -2/0, 1/1, 2/0, -1/-1, 0/2, 0/-2, -2/1, 2/2, 1/-2}
			\draw (\x,\y) node[cross]{};
		\end{tikzpicture}
		\caption{}\label{fig:Spiel4}
	\end{subfigure}
\hfill
	\begin{subfigure}[b]{.32\textwidth}
	\centering
		\begin{tikzpicture}[scale=0.7]
		\foreach \x in {-2,...,3} \draw (-2.5, \x-.5)--(2.5, \x-.5) (\x-.5, -2.5)--(\x-.5, 2.5);
		\foreach \x in {-2,...,2} 
			\node at (\x,0) [label={[yshift=-2.5cm]\footnotesize\(\x\)}] {};
		\foreach \y in {-2,...,2} 
			\node at (0,\y) [label={[xshift=2.1cm, yshift=-0.4cm]\footnotesize\(\y\)}] {};
		\foreach \x/\y in {1/0, -1/0, 0/1, 0/-1, 2/1, 1/2, -2/-2} \draw (\x,\y) circle(7pt);
		\foreach \x/\y in {-1/1, 0/0, 1/-1, -2/0, 1/1, 2/0, -1/-1, 0/2, 0/-2, -2/1, 
			2/2, 1/-2, 2/-2, -2/2, -2/-1, -1/-2} \draw (\x,\y) node[cross]{};
		\end{tikzpicture}
    	\caption{}\label{fig:Spiel5}
	\end{subfigure}
\hfill
	\begin{subfigure}[b]{.32\textwidth}
	\centering
		\begin{tikzpicture}[scale=0.7]
		\foreach \x in {-2,...,3} \draw (-2.5, \x-.5)--(2.5, \x-.5) (\x-.5, -2.5)--(\x-.5, 2.5);
		\foreach \x in {-2,...,2} 
			\node at (\x,0) [label={[yshift=-2.5cm]\footnotesize\(\x\)}] {};
		\foreach \y in {-2,...,2} 
			\node at (0,\y) [label={[xshift=2.1cm, yshift=-0.4cm]\footnotesize\(\y\)}] {};
		\foreach \x/\y in {1/0, -1/0, 0/1, 0/-1, 2/1, 1/2, -1/2, 2/-1} \draw (\x,\y) circle(7pt);
		\foreach \x/\y in {-1/1, 0/0, 1/-1, -2/0, 1/1, 2/0, -1/-1, 0/2, 0/-2, -2/1, 
			2/2, 1/-2, 2/-2, -2/2, -2/-1, -1/-2, -2/-2} \draw (\x,\y) node[cross]{};
		\end{tikzpicture}
    	\caption{}\label{fig:Spiel6}
	\end{subfigure}
	\caption{}
\end{figure}
	
	One of the two non-adjacent members of \(R_1\cap I\) 
    needs to be \((0, 1)\). For \(R_{-1}\) the same argument gives \((0, -1)\in I\) 
    and we reach the situation drawn in Figure~\ref{fig:Spiel2}.
	By symmetry about the second coordinate axis, we can suppose that the second 
    element of~\(R_1\cap I\) is \((2, 1)\) (see Figure~\ref{fig:Spiel3}). 
	Next, one of the two non-adjacent elements of \(R_2\cap I\) has to be~\((1, 2)\), 
	which brings 
    as into the situation drawn in Figure~\ref{fig:Spiel4}.   

	If \((-2,-2)\in I\), then \(I\) is contained in a union of three translates 
	of \(\langle (1, -1)\rangle\) and we are done (see Figure~\ref{fig:Spiel5}).
	So suppose \((-2,-2)\not\in I\) from now on, which causes \(I\) to be disjoint to 
	the line \(\langle (1, 1)\rangle\). If \(I\) fails to be coverable by three translates
	of this line, then \((-1, 2)\) and~\((2, -1)\) are in \(I\), and we arrive at the 
	contradiction \(|I|\le 8\) (see Figure~\ref{fig:Spiel6}).
\end{proof}

The four boxes, where the function~\(f_\alpha\) from Figure~\ref{fig:fA} attains the 
value~\(2\), are on a common line through the origin. In cases where we reach the 
outcome \(f\cong f_\alpha\) of Proposition~\ref{prop:fish}, the next lemma will help 
us to identify these values~\(2\).

\begin{lemma}\label{lem:pedantic}
	If \(f\) denotes a fishy function and a point \(P\in\FF_5^2\) satisfies \(f(P), f(2P)>0\), 
	then 
	\[
		f(-2P)+f(-P)+f(P)+f(2P)\le 8\,.
	\] 
	Moreover, equality can only hold if \(f(-2P)=f(-P)=f(P)=f(2P)=2\).
\end{lemma}

\begin{proof}
	Suppose first that \(f(\lambda P)\) is positive for every \(\lambda\in\FF_5^\times\).
	Four applications of~\ref{it:fish3} with~\({\eps=1}\) reveal
	\begin{align*}
		f(P)+f(2P)+f(-2P)&\le 6\,, \\
		f(-P)+f(2P)+f(P)&\le 6\,, \\
		f(P)+f(-2P)+f(-P)&\le 6\,, \\
		f(-P)+f(-2P)+f(2P)&\le 6\,. 
	\end{align*}
	By adding these estimates and dividing by \(3\) we obtain
	\(f(-2P)+f(-P)+f(P)+f(2P)\le 8\). Moreover, if this holds 
	with equality, then the previous four inequalities hold with 
	equality as well and \(f(-2P)=f(-P)=f(P)=f(2P)=2\) follows. 
	
	If exactly one of \(f(-P)\), \(f(-2P)\) vanishes, our claim follows from the 
	first or second of the above estimates. In the remaining case, \(f(-P)=f(-2P)=0\),  
	we exploit that~\ref{it:fish3} implies~\({2f(P)+f(2P)\le 6}\). 
\end{proof}

We call \(\FF_5^2=\{P_1, \dots, P_{25}\}\) a {\it nonincreasing enumeration} 
for a function \(f\colon \FF_5^2\lra\RR_{\ge 0}\) 
if 
\[
	f(P_1)\ge f(P_2)\ge \dots\ge f(P_{25})\,.
\]
The two lemmata that follow will allow us to derive upper bounds 
on \(\|f\|_1=\sum_{i=1}^{25} f(P_i)\) for certain fishy functions~\(f\). 

\begin{lemma}\label{lem:P4P5}
    Let \(\FF_5^2=\{P_1, \dots, P_{25}\}\) be a nonincreasing enumeration for
     a fishy function~\(f\). 
    If the support of \(f\) cannot be covered by three parallel lines
    and \(f\) fails to be isomorphic to \(f_\alpha\), then
    \begin{enumerate}[label=\rmlabel]  
    	\item\label{it:P45i} \(f(P_4)+f(P_5)\le 5\),
    	\item\label{it:P45ii} \(f(P_1)+\dots+f(P_5)\le 20\),
		\item\label{it:P45iii} and \(f(P_1)+\dots+f(P_6)<22.5\).
	\end{enumerate}
\end{lemma}

\begin{proof}
    It suffices to prove part~\ref{it:P45i}, because then 
    \(f(P_1)+f(P_2)+f(P_3)\le 3\|f\|_\infty\le 15\) yields~\ref{it:P45ii}. 
    Moreover \(f(P_6)\le f(P_5)\le \frac 12\bigl(f(P_4)+f(P_5)\bigr)=2.5\) 
    and~\ref{it:fish2} lead to~\({f(P_6)<2.5}\), 
    so that~\ref{it:P45iii} follows as well.  
    
    Assume for the sake of contradiction that \(f(P_4)+f(P_5)>5\). 
    By~\ref{it:fish3} applied to~\(P=Q=P_4\) 
    and \(\eps=-1\) we have \(f(0, 0)=0\). Similarly, 
    \begin{equation}\label{eq:1307}
    	\text{ if \(i, j\in [5]\) are distinct, then \(f(P_i\pm P_j)=0\)}\,.
	\end{equation}
    
    In particular, the set \(S=\{P_1, \dots, P_5\}\) is sum-free.  
    By Lemma~\ref{3collin} we can therefore suppose without loss of generality 
    that \((-1,1), (0,1), (1,1)\in S\). 
    In Figure~\ref{fig:P4P5.1} elements of~\(S\) are drawn as circles and 
    boxes~\(P\) satisfying \(f(P)=0\) are marked by crosses.  

\begin{figure}[h]
	\centering
	\begin{subfigure}[b]{.32\textwidth}
	\centering
		\begin{tikzpicture}[scale=0.7]
		\foreach \x in {-2,...,3} \draw (-2.5, \x-.5)--(2.5, \x-.5) (\x-.5, -2.5)--(\x-.5, 2.5);
		\foreach \x in {-2,...,2} \draw (\x,0) node[cross]{};
		\foreach \x in {-1,...,1} \draw (\x,2) node[cross]{};
		\foreach \x in {-1,...,1} \draw (\x,1) circle (7pt);
		\end{tikzpicture}
		\caption{}\label{fig:P4P5.1}
	\end{subfigure}
\hfill
	\begin{subfigure}[b]{.32\textwidth}
	\centering
		\begin{tikzpicture}[scale=0.7]
		\foreach \x in {-2,...,3} \draw (-2.5, \x-.5)--(2.5, \x-.5) (\x-.5, -2.5)--(\x-.5, 2.5);
		\foreach \x in {-2,...,2} \draw (\x,0) node[cross]{};
		\foreach \x in {-2,...,1} \draw (\x,2) node[cross]{};
		\foreach \x in {-1,0} \draw (\x,1) circle(7pt);
		\draw (1,1) node{\(P_5\)};
		\draw (2,2) node{\tiny \(>0\)};
		\draw (0,-2) node{$\xcancel{T}$};
		\draw (2,-2) node{$\xcancel{T}$};
		\end{tikzpicture}
	\caption{}\label{fig:P4P5.2}
	\end{subfigure}
\hfill
	\begin{subfigure}[b]{.32\textwidth}
	\centering
		\begin{tikzpicture}[scale=0.7]
		\foreach \x in {-2,...,3} \draw (-2.5, \x-.5)--(2.5, \x-.5) (\x-.5, -2.5)--(\x-.5, 2.5);
		\foreach \x in {-2,...,2} \draw (\x,0) node[cross]{};
		\foreach \x in {-2,...,1} \draw (\x,2) node[cross]{};
		\foreach \x in {-1,0} \draw (\x,1) circle(7pt);
		\draw (1,1) node{\(P_5\)};
		\draw (2,2) node{\tiny \(>0\)};
		\foreach \x/\y in {-2/-1, -1/-1, 2/-1, -2/-2, 0/-2, 1/-2, 2/-2, -2/1, 2/1}
			\draw (\x,\y) node{$\xcancel{T}$};
		\end{tikzpicture}
	\caption{}\label{fig:P4P5.3}
	\end{subfigure}
	\caption{}\label{fig:P4P51}
\end{figure}

We now start to analyse the set \(T=\{P_1,\dots, P_4\}\). 
	By \(f(P_4)\ge \frac12\bigl(f(P_4)+f(P_5)\bigr)>2.5\) and~\ref{it:fish3}
	we have
	\begin{equation}\label{eq:1109}
		f(2P)=0 \text{ for each } P\in T\,.
	\end{equation}
	Thus \(-2P\in T\) would imply \(f(P)=f\bigl(2(-2P)\bigr)=0\), which is absurd.  
	This proves
	\begin{equation}\label{eq:1559}	
		-2P\not\in T \text{ whenever } P\in T\,.
	\end{equation}
	
	Since the support of \(f\) cannot be covered by three rows, we can assume by symmetry 
	that~\({f(2, 2)>0}\). This implies \((1, 1)\not\in T\), so that only the possibility
	\((1, 1)=P_5\) remains. Thus \((-1, 1), (0, 1)\in T\) and~\eqref{eq:1559} 
	entails \((0, -2), (2, -2)\not\in T\). 
	The current situation is shown in Figure~\ref{fig:P4P5.2}.
	If one of the seven points 
	\[
		(-2, 1), (2, 1), (-2,-1), (-1,-1), (2,-1), (-2,-2), (1, -2)
	\]
	belonged to \(T\), then~\eqref{eq:1307} would lead to the contradiction \(f(2,2)=0\). 
	In Figure~\ref{fig:P4P5.3} these seven points are marked by a crossed out~\(T\).      
	So there remain only three candidates for the remaining two elements of~\(T\).       

\medskip
   
    {\it \hskip2em First Case: \((-1,-2)\in T\)}
        
    \smallskip

    Due to~\eqref{eq:1307} and \((-1,-2)+(1,1)=(0,-1)\) we have \(f(0, -1)=0\), 
    so \(T\) is given by the circles in Figure~\ref{fig:P4P5case1.1}.
    The remaining crosses in this figure follow from~\eqref{eq:1307} and~\eqref{eq:1109}. 
    So the two parallel lines \(\langle (2,1)\rangle\) and \((1,0)+\langle (2,1)\rangle\)
    are disjoint to the support of~\(f\). Consequently, the support of~\(f\) can be 
    covered by three parallel lines.
        
\begin{figure}[h]
	\centering
	\begin{subfigure}[b]{.32\textwidth}
	\centering
		\begin{tikzpicture}[scale=0.7]
		\foreach \x/\y in {-1/2, 2/1, 0/0, -2/-1, 1/-2}
			\filldraw[SeaGreen] (\x,\y) +(-.5,-.5) rectangle ++(.5,.5);
		\foreach \x/\y in {0/2, -2/1, 1/0, -1/-1, 2/-2}
			\filldraw[Red] (\x,\y) +(-.5,-.5) rectangle ++(.5,.5);
		\foreach \x in {-2,...,3} \draw (-2.5, \x-.5)--(2.5, \x-.5) (\x-.5, -2.5)--(\x-.5, 2.5);
		\foreach \x/\y in {-2/2, -1/2, 0/2, 1/2, -2/1, 2/1, -2/0, 
			-1/0, 0/0, 1/0, 2/0, -2/-1, -1/-1, 1/-2, 2/-2}
			\draw (\x,\y) node[cross]{};
		\foreach \x/\y in {-1/1, 0/1, -1/-2, 1/-1}\draw (\x,\y) circle(7pt);
		\draw (1,1) node{\(P_5\)};
		\draw (2,2) node{\tiny \(>0\)};
		\end{tikzpicture}
	\caption{}\label{fig:P4P5case1.1}
	\end{subfigure}
\hfill
	\begin{subfigure}[b]{.32\textwidth}
	\centering
		\begin{tikzpicture}[scale=0.7]
		\foreach \x in {-2,...,3} \draw (-2.5, \x-.5)--(2.5, \x-.5) (\x-.5, -2.5)--(\x-.5, 2.5);
		\foreach \x in {-2,...,2} \draw (\x,0) node[cross]{};
		\foreach \x in {-2,...,1} \draw (\x,2) node[cross]{};
		\foreach \x in {-1,...,2} \draw (\x,-2) node[cross]{};
		\foreach \x/\y in {-1/1, 0/1, 0/-1, 1/-1} \draw (\x,\y) circle(7pt);
		\draw (1,1) node{\(P_5\)};
		\draw (2,2) node{\tiny \(>0\)};
		\end{tikzpicture}
		\caption{}\label{fig:P4P5case2.1}
	\end{subfigure}
\hfill
	\begin{subfigure}[b]{.32\textwidth}
	\centering
		\begin{tikzpicture}[scale=0.7]
		\foreach \x in {-2,...,3} \draw (-2.5, \x-.5)--(2.5, \x-.5) (\x-.5, -2.5)--(\x-.5, 2.5);
		\foreach \x in {-2,...,2} \draw (\x,0) node[cross]{};
		\foreach \x in {-2,...,1} \draw (\x,2) node[cross]{};
		\foreach \x in {-1,...,2} \draw (\x,-2) node[cross]{};
		\foreach \x/\y in {-2/1, 2/1, -2/-1, 2/-1} \draw (\x,\y) node[cross]{};
		\foreach \x in {-2, -1, 1, 2} \draw (\x,\x) node{\(2\)};
		\foreach \x/\y in {-1/1, 0/1, 0/-1, 1/-1} \draw (\x,\y) node{\(5\)};
		\end{tikzpicture}
	\caption{}\label{fig:P4P5case1.2}
	\end{subfigure}
	\caption{}
\end{figure}

   {\it \hskip2em Second Case: \((-1,-2)\not\in T\)}
        
    \smallskip

    Now \((0, -1)\) and \((1, -1)\) are the two remaining elements of \(T\) (recall 
    Figure~\ref{fig:P4P5.3}).
    As indicated in Figure~\ref{fig:P4P5case2.1},~\eqref{eq:1307} excludes 
    some further points from the support of \(f\). 
    In view of~\ref{it:fish3}, \(\|f\|_\infty\le 5\), and \(f(2, 2)>0\) 
    we have
    \begin{align*}
        f(2,1)+f(0,-1)&\le 5\,,\\
        f(-2,1)+f(1,-1)&\le 5\,,\\
        f(-1,1)+f(2,-1)&\le 5\,,\\
        f(0,1)+f(-2,-1)&\le 5\,.
    \end{align*}
        
    Furthermore, Lemma~\ref{lem:pedantic} reveals
    \[
    	f(1,1)+f(2,2)+f(-2,-2)+f(-1,-1)\le 8\,.
	\]
	Adding these estimates and taking our knowledge about the support of \(f\) 
	into account we infer \(\|f\|_1\le 4\cdot 5+8=28\). So equality needs to hold throughout.
	
	In particular, Lemma~\ref{lem:pedantic} reveals \(f(1,1)=f(2,2)=f(-2,-2)=f(-1,-1)=2\).
	We also know \(f(2,1)+f(0,-1)=5\), whence \(f(2, 1)+f(0, -1)+\lceil f(2, 2)\rceil =7>6\).
	Due to~\ref{it:fish3} this yields \(f(2, 1)=0\) and, hence, 
	\(f(0, -1)=5\). Repeating this argument three further times we see that \(f\) is the 
	function visualised in Figure~\ref{fig:P4P5case1.2}.
	Up to an appropriate automorphism of~\(\FF_5^2\) this function agrees with \(f_\alpha\)
	(see Figure~\ref{fig:fA}).
\end{proof}

\begin{lemma}\label{lem:albern}
	Given a fishy function \(f\colon \FF_5^2\lra\RR_{\ge 0}\) set 
	\[
		S=\bigl\{P\in \FF_5^2\colon f(P)>2\bigr\}
		\quad \text{ and } \quad 
		T=\bigl\{P\in \FF_5^2\colon f(P)>1\bigr\}\,.
	\]
	If \(|S|\ge 6\) and \(f(T)>25\), then the support of \(f\) is 
	contained in three parallel lines.
\end{lemma}

\begin{proof}
	Since \(f\) is fishy, we have
    \begin{equation}\label{eq:2245}
    	T\cap (S+S)=T\cap (S-S)=\varnothing\,.
	\end{equation} 
    In particular, the set \(S\) is sum-free and by Theorem~\ref{5n=2}
	it has to be normal. Consequently, we can suppose \(S\subseteq R_1\cup R_{-1}\). 
	The sets 
	\[
		A=\bigl\{x\in \FF_5\colon (x, 1)\in S\bigr\} 
		\quad \text{ and } \quad 
		B=\bigl\{x\in \FF_5\colon (x, -1)\in S\bigr\} 
	\]
	satisfy \(|A|+|B|=|S|\ge 6\) which by Corollary~\ref{Bdumm} implies 
	\[
		A-B=(A-A)\cup (B-B)=B-A=\FF_5\,.
	\]
	Together with~\eqref{eq:2245} this leads to \(T\subseteq R_1\cup R_{-1}\). 
	Now for each \(x\in \FF_5\) the estimate 
	\[
		\sum_{i\in \FF_5}\bigl(f(i, 1)+f(i-x, -1)\bigr)=f(R_1)+f(R_{-1})\ge f(T)>25
	\]
	leads to some \(i\in \FF_5\) such that \(f(i, 1)+f(i-x, -1)>5\).
	Because of~\({(i,1)-(i-x,-1)=(x,2)}\) and~\ref{it:fish3} this implies \(f(x,2)=0\). 
	As \(x\) was arbitrary, this shows that the support of \(f\) is disjoint to \(R_2\) 
	and, for the same reason, it is disjoint to \(R_{-2}\) as well. In other words, 
	the support of \(f\) is contained in the union of the three parallel 
	lines~\(R_{-1}\),~\(R_0\), and~\(R_1\).
\end{proof}

After these preparations, we can establish Proposition~\ref{prop:fish} 
under the additional assumption~\({\|f\|_\infty>4}\).

\begin{lemma}\label{lem:fish4}
	If \(f\) is fishy, \(\|f\|_{\infty}>4\), and the support of \(f\) 
	cannot be covered by three parallel lines, then \(f\) is isomorphic to \(f_\alpha\), 
	i.e., we have \(f=f_\alpha\circ \phi\) for some automorphism \(\phi\) of \(\FF_5^2\).  
\end{lemma}

\begin{proof}
Arguing indirectly we assume that \(f\) is a counterexample. Fix a nonincreasing 
enumeration \(\FF_5^2=\{P_1, \dots, P_{25}\}\) and denote the support of~\(f\)
by~\(I\). Without loss of generality we can suppose \(P_1=(1,0)\). So \(f(1, 0)>4\),
and~\ref{it:fish3} implies for every point \((x,y)\in\FF_5^2\) that 
\begin{equation}\label{eq:nachbarngrosserfall}
    \text{if \((x,y), (x+1,y)\in I\), then \(f(x,y)+f(x+1,y)\le 1\)}\,.
\end{equation}
	
In particular, we have \((0, 0), (2, 0)\not\in I\). For each row~\(R_j\)  
let \(\rho_j=\max\{f(i, j)\colon i\in \FF_5\}\) 
be the largest value \(f\) attains on that row. Using~\eqref{eq:nachbarngrosserfall}
one easily checks for every \(j\in\FF_5\) that 

\begin{table}[h]
    \centering
    \begin{tabular}{c|c}
        if & then \\ 
        \hline     
        \(|R_j\cap I|=1\) & \(f(R_j)\le \rho_j\le 5\), \\
        \(|R_j\cap I|=2\) & \(f(R_j)\le 2\rho_j\le 10\), \\
        \(|R_j\cap I|=3\) & \(f(R_j)\le \rho_j+1\le 6\), \\
        \(|R_j\cap I|=4\) & \(f(R_j)\le 2\), \\
        \(|R_j\cap I|=5\) & \(f(R_j)<2.5\). 
    \end{tabular}
    \caption{Upper bounds on \(f(R_j)\) if \(f(1, 0)>4\)}
    \label{tab:grosserFall}
\end{table}

Notice that for the last row averaging~\eqref{eq:nachbarngrosserfall} over 
all \(x\in \FF_5\) only yields \(f(R_j)\le 2.5\), but due to~\ref{it:fish2} 
this cannot hold with equality. In fact, such rows can be ruled out entirely. 

\begin{claim}\label{no5gross}
    There is no row completely contained in \(I\). 
\end{claim}

\begin{proof}
    Due to \((0, 0)\not\in I\) we have \(R_0\not\subseteq I\). So if our claim fails, we can 
    assume, without loss of generality, that \(R_2\subseteq I\). 
	Because of \(2R_1=R_2\) and \((1, 0)-R_{-2}=R_2\) we have \(\rho_1<2.5\) 
	as well as \(\rho_{-2}<1\). So in view of Table~\ref{tab:grosserFall} we 
	obtain the upper bounds 
	\[
		f(R_{-2})<2.5\,, \quad f(R_1)<5\,, \quad \text{ and } \quad f(R_2)<2.5\,.
	\]
	They lead to the lower bounds 
	\[
		\min\bigl\{f(R_0), f(R_{-1})\bigr\}\ge 28-(2.5+10+5+2.5)=8
	\]
	and 
	\[
		f(R_1)\ge 28-(2.5+10+10+2.5)=3\,,
    \]
    which by another application of Table~\ref{tab:grosserFall} 
    imply \(|R_0\cap I|=|R_{-1}\cap I|=2\) and \(|R_1\cap I|\in \{1, 2, 3\}\). 
    But now 
    \begin{align*}
        23&<28-f(R_2)-f(R_{-2})\\
        &\le f(R_0)+f(R_{-1})+f(R_1)\\
        &\le \max\{f(P_1)+\dots+f(P_6),f(P_1)+\dots+f(P_5)+1\}
    \end{align*}
    contradicts Lemma~\ref{lem:P4P5}.
\end{proof}

Now we will study the subset \(T=\bigl\{Q\in\FF_5^2\colon f(Q)>1\bigr\}\) of \(I\). 

\begin{claim}\label{grosserfall}
    We have \(18+|T|\le f(T)\).
\end{claim}

\begin{proof}
    A simple case distinction using~\eqref{eq:nachbarngrosserfall} and 
    Claim \ref{no5gross} discloses  
    \[
    	f(R_j)+|T\cap R_j|\le 2+f(T\cap R_j)
	\]
    for every \(j\in\FF_5\). Summing this over all five rows we obtain
    \[
    	28+|T|\le \|f\|_1+|T|\le 10+f(T)\,,
	\]
    and our claim follows.   
\end{proof}
    
\begin{claim}\label{T>5gross}
    We have \(|T|\ge 8\) and \(f(T)\ge 26\).
\end{claim}
 
\begin{proof}
	The previous claim and \(\|f\|_\infty\le 5\) imply \(18+|T|\le 5|T|\), whence \(|T|\ge 5\).
	This allows us to estimate the sum \(f(T)\) more carefully with the help 
	of Lemma~\ref{lem:P4P5}, so that we arrive at 
	\[
		18+|T|
		\le 
		\bigl(f(P_1)+\dots+f(P_5)\bigr)+(|T|-5)f(P_6)
		\le 
		20+2.5(|T|-5)\,,
	\]
	i.e., \(|T|\ge 7\). However, \(|T|=7\) would require that 
	we have the equality~\({f(P_6)=2.5}\), which contradicts~\ref{it:fish2}. 
	Finally, \(|T|\ge 8\) 
	and Claim~\ref{grosserfall} entail~\({f(T)\ge 26}\).
\end{proof}

Now Lemma~\ref{lem:albern} yields \(f(P_6)\le 2\), which enables us to refine 
our analysis of Claim~\ref{grosserfall} even further. We thereby find  
\[
	18+|T|\le f(T)\le 2\cdot 5+f(P_3)+5+2(|T|-5)=5+2|T|+f(P_3)\,, 
\]
whence \(f(P_3)\ge 13-|T|\). Since \(P_2\) and \(P_3\) are in \(T\), they cannot be 
in \(R_0\) simultaneously, so we can assume without loss of generality that 
\begin{equation}\label{eq:0052}
	f(0, 1)\ge 13-|T|\,.
\end{equation}

We distinguish three cases according to the cardinality of \(T\). 

\medskip

{\it \hskip2em First Case: $|T|=8$}
          
\smallskip

Now~\eqref{eq:0052} and \(\|f\|_\infty\le 5\) lead to \(f(0,1)=f(1,0)=5\). 
Consequently, no two boxes in \(I\) can be adjacent vertically or horizontally. 
Moreover, \(f(T)\le 5+2|T|+f(P_3)=26<28\le f(I)\) 
yields \(|I|\ge |T|+1\ge 9\). So by Lemma~\ref{lem:Spiel} \(I\) can be covered by three 
parallel lines and we are done. 

\medskip

{\it \hskip2em Second Case: $|T|=9$}
        
\smallskip

This time~\eqref{eq:0052} reads \(f(0, 1)\ge 4\). Recall that we started 
with \(f(1, 0)\ge 4\). Owing to~\ref{it:fish3} this implies that a box that is 
vertically or horizontally adjacent to a box in \(T\) cannot belong to \(I\). 
In particular, each column and each row contains at most two boxes from \(T\). 
Due to~\({|T|=9}\) there are four rows and four columns each of which contain exactly 
two boxes from \(T\) and no further boxes from \(I\). Consequently, all boxes from \(I\sm T\)
need to be in the remaining column and in the remaining row, whence \(|I\sm T|\le 1\).   
It follows that no two boxes from \(I\) are vertically or horizontally adjacent 
and Lemma~\ref{lem:Spiel} tells us again that \(I\) can be covered by three parallel lines. 

\medskip

{\it \hskip2em Third Case: $|T|\ge 10$}
        
\smallskip

Since no box horizontally adjacent to a box in \(T\) can belong to \(I\), this is 
only possible if~\({T=I}\) consists of exactly \(10\) boxes, two from each row. 
Moreover,~\eqref{eq:0052} yields \(f(0, 1)\ge 3\). Together with \(T=I\) 
and~\ref{it:fish3} this implies that no two boxes in \(I\) 
can be vertically adjacent, and a final application of Lemma~\ref{lem:Spiel} shows 
that \(I\) is contained in the union of three parallel lines. 
\end{proof}

The next result explains how the function \(f_\gamma\) enters our classification. 

\begin{lemma}\label{lem:gamma}
	If \(f\colon \FF_5^2\lra\RR_{\ge 0}\) is fishy, \(3<\|f\|_\infty\le 4\) and there 
	are at least four points \(P\in\FF_5^2\) with~\(f(P)\ge 3\), then either the support 
	of \(f\) can can be covered by three parallel lines or \(f\) is isomorphic 
	to~\(f_\gamma\). 
\end{lemma}

\begin{proof}
	By Lemma~\ref{lem:P4P5} the set \(U=\{P\in \FF_5^2\colon f(P)\ge 3\}\) has size at 
	most four, so only the case~\({|U|=4}\) is interesting. Suppose without loss of 
	generality that \(f(1, 0)=\|f\|_\infty>3\), which entails that \((0, 0)\) and \((2, 0)\)
	are not in the support \(I\) of \(f\), and \((-2, 0)\not\in U\). Moreover,~\ref{it:fish3} 
	implies 
	\begin{equation}\label{eq:2112}
		f(x, y)+f(x+1, y)\le 2 \quad \text{ whenever } f(x, y), f(x+1, y)>0\,.
	\end{equation}
	
	\begin{claim}
		No row is completely contained in \(I\). 
	\end{claim}
	
	\begin{proof}
		Assume for the sake of contradiction that \(R_{-2}\subseteq I\). Due to 
		\(2R_{-1}=R_{-2}\), \(R_0-R_2=R_{-2}\), and \(R_{-2}+R_0=R_{-2}\)
		we have \(U\subseteq R_0\cup R_1\). Since no two boxes in \(U\) are horizontally 
		adjacent, we can suppose without loss of generality 
		that \(U=\{(-1, 0), (1, 0), (-1, 1), (1, 1)\}\).  
		
		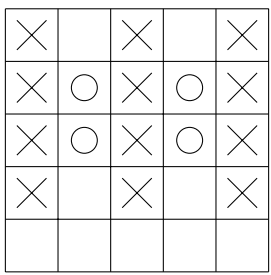
\begin{figure}[h]
    \centering
    \begin{tikzpicture}[scale=0.7]
		\foreach \x in {-2,...,3} \draw (-2.5, \x-.5)--(2.5, \x-.5) (\x-.5, -2.5)--(\x-.5, 2.5);
		\foreach \x in {-2,0,2} \foreach \y in {-1,...,2} \draw (\x,\y) node[cross]{};
		\foreach \x in {-1,1} \foreach \y in {0,1} \draw (\x,\y) circle (7pt);
	\end{tikzpicture}
	\caption{Circles indicate boxes \(Q\) with \(f(Q)>2.5\).}\label{fig:klfall}
    \end{figure}
    	
		As indicated in Figure~\ref{fig:klfall}, this excludes twelve points from \(I\). 
		Due to \(\|f\|_\infty\le 5\), \(f(2,-2)>0\), and~\ref{it:fish3} we have 
		\begin{align*} 
			f(1,0)+f(-1,2) &\le 5\,.
			\intertext{Arguing similarly, we also find}
			f(1,-1)+f(-1,1) &\le 5\,, \\
			f(-1,0)+f(1,2) &\le 5\,, \\
			f(-1,-1)+f(1,1) &\le 5\,.
		\end{align*}
		Moreover,~\eqref{eq:2112} yields \(f(R_{-2})\le 5\) by averaging. 
		Adding all these estimates, we derive the contradiction 
		\(28\le \|f\|_1\le 5\cdot 5=25\).
	\end{proof}
	
	Assuming from now on that \(I\) cannot be covered by three 
	lines, Lemma~\ref{lem:P4P5} tells us \(f(P)\le 2\) for every \(P\in \FF_5^2\sm U\). 
	Now a quick case analysis using \eqref{eq:2112} discloses 
	\begin{equation}\label{eq:2116}
		f(R_j)\le 2|R_j\cap U|+4
	\end{equation}
	for every \(j\in\FF_5\). Summing this over all \(j\in\FF_5\) we infer 
	\[
		28\le \|f\|_1\le 2|U|+20=28\,.
	\]
	Thus~\eqref{eq:2116} holds with equality for every \(j\in\FF_5\). 
	This requires \(f(P)=4\) for every \(P\in U\) and another application 
	of Lemma~\ref{lem:P4P5} discloses \(f(P)\le 1\) for all \(P\in \FF_5^2\sm U\).   
	Consequently, the five values \(f\) attains in each row are, in some 
	order, either 
	\begin{enumerate}[label=\rmlabel]
		\item\label{it:fb1} \(4\), \(4\), \(0\), \(0\), \(0\),
		\item\label{it:fb2} \(4\), \(1\), \(1\), \(0\), \(0\),
		\item\label{it:fb3} or \(1\), \(1\), \(1\), \(1\), \(0\).
	\end{enumerate}
	
	\begin{figure}[h]
	\centering
	\begin{subfigure}[b]{.32\textwidth}
	\centering
		\begin{tikzpicture}[scale=0.7]
		\foreach \x in {-2,...,3} \draw (-2.5, \x-.5)--(2.5, \x-.5) (\x-.5, -2.5)--(\x-.5, 2.5);
		\foreach \x in {-2,-1,1,2} \draw (\x,2) node{\(1\)};
		\foreach \x/\y in {0/2,0/0,2/0} \draw (\x,\y) node[cross]{};
		\draw (1,0) node{\(4\)};
		\end{tikzpicture}
		\caption{}\label{fig:gamma1}
	\end{subfigure}
\hfill
	\begin{subfigure}[b]{.32\textwidth}
	\centering
		\begin{tikzpicture}[scale=0.7]
		\foreach \x in {-2,...,3} \draw (-2.5, \x-.5)--(2.5, \x-.5) (\x-.5, -2.5)--(\x-.5, 2.5);
		\foreach \x in {-2,-1,1,2} \draw (\x,2) node{\(1\)};
		\foreach \x in {-2,2} \draw (\x,1) node{\(1\)};
		\foreach \x/\y in {0/1, 1/0} \draw (\x,\y) node{\(4\)};
		\foreach \x/\y in {0/2,0/0,2/0,-1/1,1/1,1/-1} \draw (\x,\y) node[cross]{};
		\end{tikzpicture}
	\caption{}\label{fig:gamma2}
	\end{subfigure}
\hfill
	\begin{subfigure}[b]{.32\textwidth}
	\centering
		\begin{tikzpicture}[scale=0.7]
		\foreach \x in {-2,...,3} \draw (-2.5, \x-.5)--(2.5, \x-.5) (\x-.5, -2.5)--(\x-.5, 2.5);
		\foreach \x/\y in {0/2,0/0,2/0,-1/1,1/1,1/-1,-2/0,0/-2,-1/-1} 
			\draw (\x,\y) node[cross]{};
		\foreach \x in {-2,-1,1,2} \foreach \y in {-2,2} \draw (\x,\y) node{\(1\)};
		\foreach \x in {-2,2} \foreach \y in {-1,1} \draw (\x,\y) node{\(1\)};
		\foreach \x/\y in {1/0,-1/0,0/1,0/-1} \draw (\x,\y) node{\(4\)};
		\end{tikzpicture}
	\caption{}\label{fig:gamma3}
	\end{subfigure}
	\caption{}
\end{figure}

	Because of~\(|U|=4\), the set \(M\) of all indices \(j\in \FF_5\) such 
	that~\(R_j\) is of type~\ref{it:fb3} satisfies~\({1\le |M|\le 3}\).
	Together with \(0\not\in M\) this leads to some \(j\in \FF_5\sm M\) such that \(2j\in M\).  
	We can suppose without loss of generality that \(j=1\) has these properties 
	and, moreover, that~\({f(0, 2)=0}\). Now the restriction of~\(f\) to~\(R_2\)
	is the function we see in Figure~\ref{fig:gamma1}.	
	
	For every \(x\in \FF_5^\times\) it follows from \(f(2x, 2)>0\) that \(f(x,1)<2.5\).
	Since~\(R_1\) is not of type~\ref{it:fb3}, this leaves only one possibility 
	for the restriction of \(f\) to \(R_1\), drawn in Figure~\ref{fig:gamma2}. 
	Using~\ref{it:fish3} all but four 
	boxes can now be shown not to be in \(U\) and only the possibility 
	\[
		U=\{(1, 0), (-1, 0), (0, 1), (0, -1)\}
	\] 
	remains. It is now clear that~\(f\) vanishes on the boxes marked by crosses in 
	Figure~\ref{fig:gamma3} and~\({f=f_\gamma}\) follows.  
\end{proof}

We are now sufficiently prepared for the main result of this section. 

\begin{proof}[Proof of Proposition~\ref{prop:fish}]
	Let \(f\colon \FF^2_5\lra\RR_{\ge 0}\) be a fishy function with \(\|f\|_\infty>3\) 
	whose support~\(I\) cannot be covered by three parallel lines and which is not isomorphic 
	to \(f_\alpha\). We need to show that \(f\) is isomorphic either to~\(f_\beta\) or 
	to~\(f_\gamma\). 
	
	Without loss generality we can suppose 
	\[
		f(1, 0)>3\,.
	\] 
	We keep more flexibility, however, if we do not confine ourselves 
	to the more restrictive assumption \(f(1, 0)=\|f\|_\infty\) at this 
	point. Lemma~\ref{lem:fish4} discloses
	\begin{equation}\label{eq:2320}
		\|f\|_\infty\le 4\,.
	\end{equation}
	  
	Now~\ref{it:fish3} tells us for every point \((x, y)\in\FF_5^2\) that
	\begin{equation}\label{eq:nachbarnmittlererfall}
    	\text{ if \((x,y), (x+1,y)\in I\), then \(f(x,y)+f(x+1,y)\le 2\)}\,,
	\end{equation}
	which yields \((0, 0), (2, 0)\not\in I\). Setting 
	\(\rho_j=\max\{f(i, j)\colon i\in \FF_5\}\) for 
	every \(j\in \FF_5\) we deduce from~\eqref{eq:2320} 
	and~\eqref{eq:nachbarnmittlererfall} that  

\begin{table}[h]
    \centering
    \begin{tabular}{c|c}
        if & then \\ 
        \hline     
        \(|R_j\cap I|=1\) & \(f(R_j)\le \rho_j\le 4\,,\) \\
        \(|R_j\cap I|=2\) & \(f(R_j)\le 2\rho_j\le 8\,,\) \\
        \(|R_j\cap I|=3\) & \(f(R_j)\le \rho_j+2\le 6\,,\) \\
        \(|R_j\cap I|=4\) & \(f(R_j)\le 4\,,\) \\
        \(|R_j\cap I|=5\) & \(f(R_j)\le 5\,.\) 
    \end{tabular}
    \caption{Upper bounds on \(f(R_j)\) if \(f(1, 0)> 3\)}
    \label{tab:mittlererFall}
\end{table}
        
\begin{claim}\label{lm18A}
    If \(Q\in\FF_5^2\) satisfies \(f(Q)>3\), then \(-2Q\not\in I\).
\end{claim}

\begin{proof}
	It suffices to show~\({(-2,0)\not\in I}\); assuming contrariwise that \(f(-2, 0)>0\), 
	we conclude 
	\[
		f(x,y)+f(x+2,y)\le 5 \text{ for all \((x,y)\in\FF_5^2\)}
	\]
	from~\ref{it:fish3}. Together with~\eqref{eq:nachbarnmittlererfall} 
	and Table~\ref{tab:mittlererFall} this shows that if \(f(R_j)>5\) holds 
	for some row~\(R_j\), then \(|R_j\cap I|=3\) and \(\rho_j>3\). Moreover, 
	even in this case we have \(f(R_j)\le 6\). 
	
	As our claim holds trivially if \(f\) is isomorphic to \(f_\gamma\), 
	Lemma~\ref{lem:gamma} tells us that the set 
	\[
		M=\{j\in \FF_5\colon f(R_j)>5\}
	\]
	has size at most three. On the other hand, 
	\[
		28\le \sum_{j\in \FF_5} f(R_j)\le 6|M|+5(5-|M|)=25+|M|
	\]
	implies \(|M|\ge 3\). Thus we have equality throughout and the 
	numbers  \(f(R_{-2}),\dots, f(R_2)\) are, in some order, \(5, 5, 6, 6, 6\).  
	Both indices \(j\) with \(f(R_j)=5\) satisfy \(\rho_j<3\) by Lemma~\ref{lem:gamma}
	and, therefore, \(|R_j\cap I|\in \{2, 5\}\). It cannot be the case that
	\(|R_j\cap I|=2\) holds for both of them, because then every row contains a box \(P\)
	with \(f(P)>2.5\), contrary to Lemma~\ref{lem:P4P5}. 
	So we can assume, without loss of generality, that \(f(R_2)=5\) and \(R_2\subseteq I\). 
	But now \(2R_1=R_2\) yields~\({\rho_1<2.5}\), which requires \(f(R_1)=5\) 
	and \(R_1\subseteq I\). Another repetition of this argument, this time 
	using \(2R_{-2}=R_1\), yields a contradiction.  
\end{proof}

\begin{claim}\label{clm:0944}
    If some row is a subset of \(I\), then \(f\) is isomorphic to \(f_\beta\) 
    (see Figure~\ref{fig:fB}).
\end{claim}

\begin{proof}
    Since \((0, 0)\not\in I\), we can suppose, without loss generality, that \(R_2\subseteq I\),
    which according to Table~\ref{tab:mittlererFall} implies \(f(R_2)\le 5\). 
    Owing to 
    \[
    	(1, 0)-R_{-2}=R_2\,, \quad 
		2R_1=R_2\,, \quad 
		-2R_{-1}=R_2\,,
	\]
	and Claim~\ref{lm18A} we have 
    \begin{equation}\label{eq:1025}    
    	\rho_{-2}<2\,, \quad 
		\rho_1<2.5\,, \quad 
		\text{ and } \quad \rho_{-1}\le 3\,,
	\end{equation}
	whence
	\[
		f(R_{-2})\le 5\,, \quad 
		f(R_1)\le 5\,, \quad 
		\text{ and } \quad 
		f(R_{-1})\le 6\,.
	\] 
	
	Improving the third bound we contend that 
	\begin{equation}\label{eq:1028} 
		f(R_{-1})\le 5\,.
	\end{equation}
	Otherwise, due to \(\rho_{-1}\le 3\), Table~\ref{tab:mittlererFall}
	yields \(|R_{-1}\cap I|=2\). Writing \(R_{-1}\cap I=\{P, Q\}\) with \(f(P)\ge f(Q)\)
	we could conclude from~\(2f(P)\ge f(P)+f(Q)>5\) that \(f(2P)=f(P+Q)=0\). 
	Since \(2P\) and \(P+Q\) are distinct points in \(R_{-2}\), this 
	shows \(|R_{-2}\cap I|\le 3\), which together with~\({\rho_{-2}<2}\) 
	entails \(f(R_{-2})<4\). Thus we get the contradiction 
	\[
		28
		\le 
		\|f\|_1
		=
		f(R_{-2})+f(R_{-1})+f(R_0)+f(R_1)+f(R_2)
		<
		4+6+8+5+5=28\,.
	\]
	
	Having thereby proved~\eqref{eq:1028} we observe that 
	\[
		28
		\le 
		\|f\|_1
		=
		f(R_{-2})+f(R_{-1})+f(R_0)+f(R_1)+f(R_2)
		\le
		5+5+8+5+5=28
	\]
	holds with equality. So in view of~\eqref{eq:1025} and Table~\ref{tab:mittlererFall}
	the rows \(R_1\) and \(R_{-2}\) are also contained in~\(I\). Because of \(2R_{-1}=R_{-2}\)
	this entails \(\rho_{-1}<2.5\), so that \(R_{-1}\subseteq I\) holds as well. 
	
	Summarising we have \(R_j\subseteq I\) and \(f(R_j)=5\) for every nonzero \(j\in \FF_5\). 
	This equality in Table~\ref{tab:mittlererFall} is only possible if \(f\) always attains 
	the value \(1\) on these four rows. Together with~\(f(R_0)=8\) this shows \(f=f_\beta\). 
\end{proof}

We can henceforth suppose 
\begin{equation}\label{eq:1057}
	R_j\not\subseteq I \text{ for every } j\in \FF_5\,.
\end{equation}
Concerning the set 
\[
	S=\bigl\{Q\in\FF_5^2\colon f(Q)>2\bigr\}
\]
this has the following consequence. 

\begin{claim}\label{mittlererfall}
	We have \(8+2|S|\leq f(S)\).
\end{claim}

\begin{proof}
    A quick case analysis based on~\eqref{eq:nachbarnmittlererfall} and~\eqref{eq:1057} shows 
    \begin{equation}\label{eq:star}
        f(R_j)+2|S\cap R_j|\le 4+f(S\cap R_j)
    \end{equation}    
    for every \(j\in\FF_5\). By summing these five inequalities we obtain
    \[
    	28+2|S|\le \|f\|_1+2|S|\le 20+f(S)\,,
	\]
    which implies our claim.
\end{proof}
 
In view of \(\|f\|_\infty\le 4 \) the previous claim implies \(8+2|S|\le f(S)\le 4|S|\), 
whence \(|S|\ge 4\). In the special case \(|S|=4\) we need to have the equality
\(f(P_1)=f(P_2)=f(P_3)=f(P_4)=4\) and Lemma~\ref{lem:gamma} implies \(f\cong f_\gamma\). 
Assuming from now on that this is not the case, we have~\(|S|\ge 5\). 
The case \(|S|=5\) itself is impossible, for then Lemma~\ref{lem:P4P5} and Claim~\ref{mittlererfall} would imply the contradiction
\(18\le f(S)\le 3\cdot 4+5=17\). More generally, we obtain    
\[
    8+2|S|\le f(S)<4+4+f(P_3)+5+2.5(|S|-5)\,,
\]
whence 
\begin{equation}\label{eq:0335}
	15-2f(P_3)< |S|\,.
\end{equation} 

In particular, we have \(|S|\ge 8\). By~\ref{it:fish3} the set \(S\) is also sum-free
and, hence, normal.  
    
\medskip

{\it \hskip2em First Case: \(3<f(P_3)\le 4\)}
        
\smallskip

	Let~\({L \subseteq \FF_5^2}\) be a line satisfying \(S\subseteq L\cup (-L)\). 
	The box principle allows us to suppose that the set \(F=L \cap\{P_1,P_2,P_3\}\) has at 
	least size~\(2\). The set $S'=\bigl(S\cup (-S)\bigr)\cap L\) 
	satisfies \(|S'|\ge \frac12 |S|\ge 4\). Due to \(|F|+|S'|\ge 6\) and Corollary~\ref{Bdumm}
	the sets \(F+S'\) and \(F-S'\)  are the entire lines \(L+L\) and \(L-L\), respectively. 
	As they need to be disjoint to \(I\), we can cover \(I\) by three translates of \(L\).

\medskip

{\it \hskip2em Second Case: \(f(P_3)\le 3\)}
        
\smallskip

    Now~\eqref{eq:0335} tells us \(|S|\ge 10\). Recall that \(f(P)>2\)
    holds for every \(P\in S\). Using~\eqref{eq:nachbarnmittlererfall} 
    we infer that \(S=I\) contains exactly two boxes from each row. 
    In particular, the normality of \(S\) implies that \(I\) consists  
    of two parallel lines. 
\end{proof}

We conclude this section with another lemma on fishy functions, that will assist us later 
when proving Propositions~\ref{prop:klein} and~\ref{prop:0135}.

\begin{lemma}\label{lem:314}
	Let \(f\colon \FF_5^2\lra \RR_{\ge 0}\) be a fishy function with support~\(I\). 
	If \(\|f\|_\infty\le 3\), then for every point \(P\in \FF_5^2\) with \(f(P)>2.8\) 
	and \(f(-P)>2.5\) there exists a point \(Q\) with 	
	\[
		Q, P+Q\in I \quad \text{  and } \quad f(Q)+f(P+Q)>2.2\,.
	\]  
\end{lemma}

\begin{proof}
	Assume contrariwise that \(f(1, 0)>2.8\),
	\(f(-1,0)>2.5\), but \(f(x, y)+f(x+1, y)\le 2.2\) for all \((x, y)\in\FF_5^2\) 
	with \((x, y), (x+1, y)\in I\). 
    Now for every \(j\in \FF_5\) we find that 
    
    \begin{table}[h]
    \centering
    \begin{tabular}{c|c}
        if & then \\ 
        \hline     
        \(|R_j\cap I|=1\) & \(f(R_j)\le 3\,,\) \\
        \(|R_j\cap I|=2\) & \(f(R_j)\le 6\,,\) \\
        \(|R_j\cap I|=3\) & \(f(R_j)\le 5.2\,,\) \\
        \(|R_j\cap I|=4\) & \(f(R_j)\le 4.4\,,\) \\
        \(|R_j\cap I|=5\) & \(f(R_j)<5.5\,.\) 
    \end{tabular}
    \caption{Maximal values for $f(R_j)$}
    \label{tab:minifall}
	\end{table}
        
    Due to \(f(R_0)\le 6\) we can suppose, without loss of generality, that 
    \[
    	f(R_1)\ge\tfrac 14\bigl(\|f\|_1-f(R_0)\bigr)\ge \tfrac14 (28-6)=5.5\,.
	\]
	By Table~\ref{tab:minifall} this yields \(|R_1\cap I|=2\) and we can assume that 
    \[
    	R_1\cap I=\{(1, 1), (-1, 1)\}\,.
	\]
    
    The four points \(Q\) marked by circles in Figure~\ref{fig:klfall} satisfy \(f(Q)>2.5\).
    Due to Lemma~\ref{lem:P4P5} they are the only such points and by~\ref{it:fish3} the 
    boxes marked by crosses in Figure~\ref{fig:klfall} are not in \(I\)---they are sums or 
    differences of boxes marked by circles. This entails~\(f(R_{-1})<5\) and~\(f(R_2)<5\), 
    and we reach the contradiction 
    \[
    	28\le \|f\|_1<3\cdot6+2\cdot 5=28\,. \qedhere
	\]
\end{proof} 
    
\section{Acceptable functions}\label{sec:klein}
This entire section deals with the proof of Proposition~\ref{prop:klein}. 
Let \(f\colon \FF_5^3\lra\RR_{\ge 0}\) denote an acceptable function with 
support~\(I\), standard projection~\(g\), and special point \(P_\star\).  
Assume for the sake of contradiction that \(f(L)\le 3\) holds for every 
line \(L\subseteq \FF_5^3\). When applied to the lines~\(L_{x, y}\) this tells 
us \(\|g\|_\infty\le 3\). Since \(g\) is fishy, we have \(g(0, 0)\le 2\). Thus the special 
point~\(P_\star\) cannot be the origin and we can suppose that it is either~\((1, 0)\)
or \((-1, 0)\). We prefer not to commit ourselves to one of these two alternatives, because 
in this manner we keep the current situation symmetric with respect to the 
reflection \(x\longmapsto -x\), which offers some advantages. 

For each point \(Q=(x, y)\in\FF_5^2\) we set 
\[
	I_Q=I_{x,y}=\bigl\{z\in\FF_5\colon (x, y, z)\in I\bigr\}\,.
\]
We will often use that~\ref{it:owl1} implies \(g(Q)\le |I_Q|\) for every \(Q\in\FF_5^2\). 
Similarly, this assumption reveals \(f(x, y, z)\ge g(x, y)-(|I_{x,y}|-1)\)
whenever \(z\in I_{x, y}\). 
Since the group of affine permutations of \(\FF_5\) acts transitively on the 
three-element subsets of~\(\FF_5\), we can suppose by~\ref{it:owl5} that 
\begin{equation}\label{eq:IPstern}
	I_{P_\star}=\{-1, 0, 1\}\,,
\end{equation}
which has the following consequence. 

\begin{claim}\label{clm:1958}
	If \(f(x, y, z)\ge 0.2\), \(\delta\in \{-1, 1\}\), and \(\eta\in \{-1, 0, 1\}\), 
	then \((x+\delta, y, z+\eta)\not\in I\). 
\end{claim}

\begin{proof}
	Write \(P_\star=(\eps, 0)\) with~\(\eps\in\{-1, 1\}\). 
	From~\ref{it:owl4} and \(\delta\eps\eta\in \{-1, 0, 1\}=I_{P_\star}\) 
	we conclude 
	\[
		f(x, y, z)+f(\eps, 0, \delta\eps\eta)\ge 0.2+\bigl(g(P_\star)-2\bigr)>1\,.
	\]
	Due to~\ref{it:owl3} applied with \(\delta\eps\) here in place of \(\eps\) 
	there it follows that the point 
	\[
		(x, y, z)+\delta\eps(\eps, 0, \delta\eps\eta)=(x+\delta, y, z+\eta)
	\]
	is indeed not in \(I\).
\end{proof}

Denoting the support of \(g\) by \(J\) we can now reformulate~\ref{it:owl6} in a more 
concrete manner. 

\begin{claim}\label{clm:2001}
	If \((x, y), (x+1, y)\in J\) and \(g(x, y)+g(x+1, y)> 2.2\), then there exists 
	some~\({a\in\FF_5}\) such that 
	\begin{itemize}
		\item either \(I_{x, y}=\{a\}\) and \(I_{x+1, y}=\{a-2, a+2\}\)
		\item or \(I_{x, y}=\{a-2, a+2\}\) and \(I_{x+1, y}=\{a\}\).
	\end{itemize}
\end{claim}

\begin{proof}
	Since \(P_\star\) is either \((-1, 0)\) or \((1, 0)\), there exists a point \(Q\) 
	such that 
	\[
		\bigl\{Q, P_\star+Q\bigr\}=\bigl\{(x, y), (x+1, y)\bigr\}\,.
	\]
	
	Now~\ref{it:owl6} tells us \(|I_{x, y}|+|I_{x+1, y}|=3\). As both summands are positive, 
	this means that either we have \(|I_{x, y}|=1\) and \(|I_{x+1, y}|=2\), or it is the 
	other way around. 
	In the first case, we write \(I_{x, y}=\{a\}\) with an appropriate element \(a\in \FF_5\). 
	Owing to \(g(x+1, y)\le 2\) we have~\({g(x, y)> 2.2-2=0.2}\)
	and, therefore, Claim~\ref{clm:1958} implies \(a-1, a, a+1\not\in I_{x+1, y}\), which in 
	turn leads to \(I_{x+1, y}=\{a-2, a+2\}\).
	The case that \(|I_{x, y}|=2\) and \(|I_{x+1, y}|=1\) is treated similarly.
\end{proof}

These claims lead to a symmetric strengthening of \eqref{eq:IPstern}.

\begin{claim}\label{clm:1420} 
	We have \(I_{1, 0}=I_{-1, 0}=\{-1, 0, 1\}\).
\end{claim}

\begin{proof}
	Owing to~\eqref{eq:IPstern} and \(|I_{-P_\star}|\ge \lceil g(-P_\star)\rceil=3\)
	it suffices to show \(-2, 2\not\in I_{-P_\star}\). 
	By Lemma~\ref{lem:314} applied to the fishy function \(g\) and to \(P=P_\star\) 
	we find some \(Q\in \FF_5^2\) satisfying~\({Q, P_\star+Q\in J}\) 
	and \(g(Q)+g(P_\star+Q)>2.2\). Due to Claim~\ref{clm:2001} there exists 
	some \(a\in\FF_5\) such that 
	\begin{itemize}
		\item either \(I_{Q}=\{a\}\) and \(I_{P_\star+Q}=\{a-2, a+2\}\)
		\item or \(I_{Q}=\{a-2, a+2\}\) and \(I_{P_\star+Q}=\{a\}\).
	\end{itemize}
	
	In the first case, we have 
	\[
		f(Q, a)+f(P_\star+Q, a\mp 2)\ge g(Q)+g(P_\star+Q)-1>1
	\]
	for both signs and, therefore,~\ref{it:owl3} yields 
	\[
		(-P_\star, \pm 2)=(Q, a)-(P_\star+Q, a\mp 2)\not\in I\,,
	\] 
	i.e., \(\pm 2\not\in I_{-P_\star}\). The second case is analogous. 
\end{proof}

Since \(g\) fulfills~\ref{it:fish3}, we have
\begin{equation}\label{eq:nachbarnkleinerfall}
    g(x,y)+g(x+1,y)\le 3 \text{ whenever } (x,y), (x+1,y)\in J\,. 
\end{equation}

Moreover, \ref{it:owl4} and~\ref{it:fish3} yield \((-2, 0), (0, 0), (2, 0)\not\in J\), 
whence \(g(R_0)\le 6\). 
Thus the average of \(g(R_j)\) over all \(j\in\FF_5\sm\{0\}\) is at least \(5.5\). 
The conjunction of our next three claims shows, however, that \(g(R_j)>5.5\) can 
only occur for some \(j\ne 0\), if~\({|R_j\cap J|=3}\). 

\begin{claim}\label{5.5_5}
    If \(R_j\subseteq J\) holds for some \(j\in \FF_5\), then \(g(R_j)<5.5\).
\end{claim}

\begin{proof}
    Assume for the sake of contradiction that \(R_1\subseteq J\) and \(g(R_1)>5.5\). 
    Without loss of generality we can suppose 
    \begin{equation}
    	g(0, 1)+g(1, 1)>2.2\,,
	\end{equation}
	so that Claim~\ref{clm:2001} gives two possibilities concerning the 
	sets \(I_{0, 1}\) and \(I_{1, 1}\). 
	By applying a suitable automorphism of \(\FF_5^3\), 
    which is either of the form \((x, y, z)\longmapsto (x, y, z-ay)\) or of the form 
    \((x, y, z)\longmapsto (-x+y, y, z-ay)\) for some \(a\in\FF_5\), we can suppose further 
    that \(I_{0, 1}=\{0\}\) and \(I_{1, 1}=\{-2, 2\}\).
    Notice that \(f(0, 1, 0), f(1, 1, 2), f(1, 1, -2)>0.2\). 
    Thus the boxes of \(R_1\) corresponding to \(x=0\) and \(x=1\) are structured as 
    in Figure~\ref{fig:5row.2}. Next, Claim~\ref{clm:1958} leads to the seven crosses 
    we see in the squares for \(x=-1\) and \(x=2\). 
    If for some \(z\in \FF_5\) we had~\({f(-2, 1, z)\ge 0.2}\), then 
    Claim~\ref{clm:1958} would yield further crosses in those two squares, so that 
    a contradiction to \(R_1\subseteq J\) would arise. This explains the five entries
    in the square for \(x=-2\).  
    
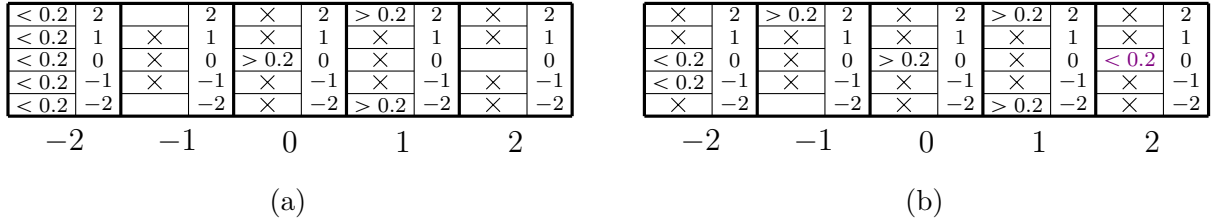
\begin{figure}[h]
	\centering
	\begin{subfigure}[b]{.49\textwidth}
	\centering
		\begin{tikzpicture}[scale=1.5]
		\foreach \x in {-2,...,3}\draw[very thick] (\x-.5, -.5)--(\x-.5, .5);
		\foreach \y in {0,1} \draw[very thick] (-2.5, \y-.5)--(2.5, \y-.5);  
    	\foreach \x in {-2,...,2} {
    		\node at (\x,0) [label={[yshift=-1.52cm]\(\x\)}] {};
    		\draw (\x+.1,-.5)--(\x+.1,.5);
			\foreach \y in {1,...,4} \draw (\x-.5,\y/5-.5)--(\x+.1, \y/5-.5);
			\foreach \y in {-2,...,2}
    			\node at  (\x,\y/5) [label={[xshift=.45cm, yshift=-0.38cm]\tiny\(\y\)}] {};
    	}
    	\foreach \x in {-1,1} \foreach \y in {-1,0,1} \draw (\x-0.2,\y/5) node[crosspt]{};
    	\foreach \x in {0,2} \foreach \y in {-2,-1,1,2} \draw (\x-0.2,\y/5) node[crosspt]{};
    	\foreach \x/\y in {0/0, 1/2, 1/-2} \draw (\x-.2, \y/5) node{\tiny {\(>0.2\)}};
	    \foreach \y in {-2,...,2} \draw (-2.2,\y/5) node{\tiny {\(< 0.2\)}};
    \end{tikzpicture}
    \caption{}\label{fig:5row.2}
\end{subfigure}
\hfill
	\begin{subfigure}[b]{.49\textwidth}
	\centering
		\begin{tikzpicture}[scale=1.5]
        \foreach \x in {-2,...,3}\draw[very thick] (\x-.5, -.5)--(\x-.5, .5);
		\foreach \y in {0,1} \draw[very thick] (-2.5, \y-.5)--(2.5, \y-.5);  
    	\foreach \x in {-2,...,2} {
    		\node at (\x,0) [label={[yshift=-1.52cm]\(\x\)}] {};
    		\draw (\x+.1,-.5)--(\x+.1,.5);
			\foreach \y in {1,...,4} \draw (\x-.5,\y/5-.5)--(\x+.1, \y/5-.5);
			\foreach \y in {-2,...,2}
    			\node at  (\x,\y/5) [label={[xshift=.45cm, yshift=-0.38cm]\tiny\(\y\)}] {};
    	}
    	\foreach \x in {-1,1} \foreach \y in {-1,0,1} \draw (\x-0.2,\y/5) node[crosspt]{};
    	\foreach \x in {0,2} \foreach \y in {-2,-1,1,2} \draw (\x-0.2,\y/5) node[crosspt]{};
		\foreach \y in {-2,1,2} \draw (-2.2,\y/5) node[crosspt]{};
		\foreach \x/\y in {0/0, 1/2, 1/-2, -1/2} \draw (\x-.2, \y/5) node{\tiny {\(>0.2\)}};
        \foreach \y in {0,-1} \draw (-2.2, \y/5) node{\tiny {\(< 0.2\)}};
    	\draw (1.8,0) node[text=violet]{\tiny{\(< 0.2\)}};
        \end{tikzpicture}
    \caption{}\label{fig:5row.3}
\end{subfigure}
	\caption{Internal structure of \(R_1\)}
\end{figure}

	So we have \(g(-2, 1)<1\) and, consequently, \(g(-1, 1)>g(R_1)-g(-2, 1)-4>0.5\). 
	Without loss of generality we can therefore suppose \(f(-1, 1, 2)>0.2\).
	As shown in Figure~\ref{fig:5row.3}, Claim~\ref{clm:1958} yields three further crosses 
	in the left square. Moreover, \(f(2, 1, 0)\ge 0.2\) is impossible, for then 
	Claim~\ref{clm:1958} would lead to the contradiction \((-2, 1)\not\in J\). 

	Observe that there are only eight boxes in Figure~\ref{fig:5row.3} without crosses. 
	Because of 
	\[
		(-2, 1, -1)-(-1, 1, -2)=(-1, 0, 1)\in I
	\]
	property~\ref{it:owl3} yields 
	\[
		f(-2, 1, -1)+f(-1, 1, -2)\le 1\,.
	\]
	So altogether we get the contradiction 
	\[
		g(R_1)< 1+4+2\cdot 0.2=5.4\,. \qedhere
	\]
\end{proof}

Using~\eqref{eq:nachbarnkleinerfall}, Claim \ref{5.5_5}, and \(\|g\|_\infty\le 3\) 
we find for every \(j\in\FF_5\) that 

\begin{table}[h]
    \centering
    \begin{tabular}{c|c}
        if & then \\ 
        \hline     
        \(|R_j\cap J|=1\) & \(f(R_j)\le 3\,,\) \\
        \(|R_j\cap J|=2\) & \(f(R_j)\le 6\,,\) \\
        \(|R_j\cap J|=3\) & \(f(R_j)\le 6\,,\) \\
        \(|R_j\cap J|=4\) & \(f(R_j)\le 6\,,\) \\
        \(|R_j\cap J|=5\) & \(f(R_j)<5.5\,.\) 
    \end{tabular}
    \caption{Maximal values for \(g(R_j)\)}
    \label{tab:kleinerFall}
	\end{table}
	
In particular, we have 

\begin{equation}\label{eq:6}
	g(R_j)\le 6 \quad \text{ for every \(j\in\FF_5\)}\,.
\end{equation}

We plan to improve the second and fourth row of our table, starting with the easier task. 

\begin{claim}\label{5.5_2}
    If \(j\in\FF_5\sm\{0\}\) and \(|R_j\cap J|=2\), then \(g(R_j)<5.5\).
\end{claim}

\begin{proof}
	The argument is very similar to the proof of Lemma~\ref{lem:314}. Assume contrariwise 
	that, without loss of generality, \(|R_1\cap J|=2\) but \(g(R_j)>5.5\). 
	We can suppose further that~\({R_1\cap J=\{(-1, 1), (1, 1)\}}\), so that we again 
	reach Figure~\ref{fig:klfall} with the same meaning of circles and crosses. 
	As there, an argument involving Lemma~\ref{lem:P4P5} discloses \(g(R_2), g(R_{-1})<5\)
	and we again reach the contradiction 
	\(
		28\le \|g\|_1<2\cdot 5+3\cdot 6=28.
	\)
\end{proof}

Rows with exactly four boxes in \(J\) are less easy to analyse. It will be convenient from 
now on to denote by \(R_j\) not only the subset \(\FF_5\times\{j\}\) of \(\FF_5^2\) but 
also the subset \(\FF_5\times\{j\}\times\FF_5 \) of \(\FF_5^3\).

\begin{claim}\label{5.5_4}
    We have \(g(R_j)<5.5\) for every \(j\in\FF_5\) with \(|R_j\cap J|=4\).
\end{claim}

\begin{proof}
    Suppose for the sake of contradiction that \(R_1\cap J=R_1\sm \{(2, 1)\}\), 
    but \(g(R_1)>5.5\). Due to \eqref{eq:nachbarnkleinerfall} this yields
    \begin{equation}\label{eq:0204}
    	g(-2,1)+g(-1,1)>2.5 
		\quad \text{ and } \quad 
		g(0,1)+g(1,1)>2.5\,.
	\end{equation}
	
	By Claim~\ref{clm:2001} and the first estimate 
    we can suppose that the boxes~\({(-2, 1)}\) and \((-1, 1)\) are structured as in 
    Figure~\ref{fig:erstemgl4.2} or~\ref{fig:zweitemgl4.1}. 
    
    \begin{figure}[h]
	\centering
	\begin{subfigure}[b]{.49\textwidth}
	\centering
		\begin{tikzpicture}[scale=1.5]
		\foreach \x in {-2,...,3}\draw[very thick] (\x-.5, -.5)--(\x-.5, .5);
		\foreach \y in {0,1} \draw[very thick] (-2.5, \y-.5)--(2.5, \y-.5);  
    	\foreach \x in {-2,...,1} {
    		\node at (\x,0) [label={[yshift=-1.52cm]\(\x\)}] {};
    		\draw (\x+.1,-.5)--(\x+.1,.5);
			\foreach \y in {1,...,4} \draw (\x-.5,\y/5-.5)--(\x+.1, \y/5-.5);
			\foreach \y in {-2,...,2}
    			\node at  (\x,\y/5) [label={[xshift=.45cm, yshift=-0.38cm]\tiny\(\y\)}] {};
    	}
		\node at (2,0) [label={[yshift=-1.52cm]\(2\)}] {};
    	\foreach \x in {-1,1} \foreach \y in {-1,0,1} \draw (\x-0.2,\y/5) node[crosspt]{};
    	\foreach \x in {-2,0} \foreach \y in {-2,-1,1,2} \draw (\x-0.2,\y/5) node[crosspt]{};
    	\foreach \x/\y in {-1/-2, -1/2, -2/0, 1/-2, 1/2, 0/0}
			\draw (\x-0.2,\y/5) node{\tiny {\(>0.5\)}};
        \draw (2,0) node[crossgr]{};
    \end{tikzpicture}
    \caption{}\label{fig:erstemgl4.2}
	\end{subfigure}
\hfill
	\begin{subfigure}[b]{.49\textwidth}
	\centering
		\begin{tikzpicture}[scale=1.5]
        \foreach \x in {-2,...,3}\draw[very thick] (\x-.5, -.5)--(\x-.5, .5);
		\foreach \y in {0,1} \draw[very thick] (-2.5, \y-.5)--(2.5, \y-.5);  
    	\foreach \x in {-2,...,1} {
    		\node at (\x,0) [label={[yshift=-1.52cm]\(\x\)}] {};
    		\draw (\x+.1,-.5)--(\x+.1,.5);
			\foreach \y in {1,...,4} \draw (\x-.5,\y/5-.5)--(\x+.1, \y/5-.5);
			\foreach \y in {-2,...,2}
    			\node at  (\x,\y/5) [label={[xshift=.45cm, yshift=-0.38cm]\tiny\(\y\)}] {};
    	}
		\node at (2,0) [label={[yshift=-1.52cm]\(2\)}] {};
    	\foreach \x in {-2,0} \foreach \y in {-1,0,1} \draw (\x-0.2,\y/5) node[crosspt]{};
    	\foreach \y in {-2,-1,1,2} \draw (-1-0.2,\y/5) node[crosspt]{};
    	\foreach \x/\y in {-2/-2,-2/2,-1/0} \draw (\x-0.2,\y/5) node{\tiny {\(>0.5\)}};
   		\draw (2,0) node[crossgr]{};
    	\end{tikzpicture}
    \caption{}\label{fig:zweitemgl4.1}
    \end{subfigure}
	\caption{}
\end{figure}
	    
    In the left case, Claim~\ref{clm:1958} yields the four crosses in the square 
    corresponding to \(x=0\). Then, Claim~\ref{clm:2001} and the second estimate 
    in~\eqref{eq:0204} give the remaining entries of Figure~\ref{fig:erstemgl4.2}.
    In Figure~\ref{fig:zweitemgl4.1}, however, Claim~\ref{clm:1958} only yields three
    crosses in the middle square and it remains unclear whether \(|I_{0, 1}|=1\) or
    \(|I_{0, 2}|=1\). Thus we get two subcases and, without loss of generality, one 
    of the two situations shown in Figure~\ref{fig:zweitemgl4.2} and~\ref{fig:drittemgl4.2}
    occurs.
    
	In Figure~\ref{fig:drittemgl4.2} the four green boxes form a subset \(A\) of 
	the line \((-2, 1, -2)+\langle (1, 0, 2)\rangle\) and thus we have \(f(A)\le 3\), 
	which yields the contradiction \(g(R_1)\le f(A)+2\le 5\). For this reason only 
	the two cases depicted in Figure~\ref{fig:erstemgl4.2} or~\ref{fig:zweitemgl4.2}
	are possible. As the automorphism~\({(x, y, z)\longmapsto (-x-y,y,z)}\) exchanges these 
	two cases and preserves our assumptions on the row \(R_0\), we can suppose without 
	loss generality that \(R_1\) is structured as in Figure~\ref{fig:erstemgl4.2}. 
    
\begin{figure}[h]
	\centering
	\begin{subfigure}[b]{.49\textwidth}
		\centering
		\begin{tikzpicture}[scale=1.5]
		\foreach \x in {-2,...,3}\draw[very thick] (\x-.5, -.5)--(\x-.5, .5);
		\foreach \y in {0,1} \draw[very thick] (-2.5, \y-.5)--(2.5, \y-.5);  
    	\foreach \x in {-2,...,1} {
    		\node at (\x,0) [label={[yshift=-1.52cm]\(\x\)}] {};
    		\draw (\x+.1,-.5)--(\x+.1,.5);
			\foreach \y in {1,...,4} \draw (\x-.5,\y/5-.5)--(\x+.1, \y/5-.5);
			\foreach \y in {-2,...,2}
    			\node at  (\x,\y/5) [label={[xshift=.45cm, yshift=-0.38cm]\tiny\(\y\)}] {};
    	}
		\node at (2,0) [label={[yshift=-1.52cm]\(2\)}] {};
    \foreach \x in {-2,0} \foreach \y in {-1,0,1} \draw (\x-0.2,\y/5) node[crosspt]{};
    \foreach \x in {-1,1} \foreach \y in {-2,-1,1,2} \draw (\x-0.2,\y/5) node[crosspt]{};
    \foreach\x/\y in {-2/-2, -2/2, -1/0, 0/-2, 0/2, 1/0} 
    	\draw (\x-0.2,\y/5) node{\tiny {\(>0.5\)}};
    \draw (2,0) node[crossgr]{};
    \end{tikzpicture}
    \caption{}\label{fig:zweitemgl4.2}
\end{subfigure}
\hfill
\begin{subfigure}[b]{.49\textwidth}
	\centering
		\begin{tikzpicture}[scale=1.5]
		\foreach \x/\y in {-2/-2, -1/0, 0/2, 1/-1}
			\filldraw[SeaGreen] (\x,\y/5) +(-.5,-.1) rectangle ++(.1,.1);
		\foreach \x in {-2,...,3}\draw[very thick] (\x-.5, -.5)--(\x-.5, .5);
		\foreach \y in {0,1} \draw[very thick] (-2.5, \y-.5)--(2.5, \y-.5);  
    	\foreach \x in {-2,...,1} {
    		\node at (\x,0) [label={[yshift=-1.52cm]\(\x\)}] {};
    		\draw (\x+.1,-.5)--(\x+.1,.5);
			\foreach \y in {1,...,4} \draw (\x-.5,\y/5-.5)--(\x+.1, \y/5-.5);
			\foreach \y in {-2,...,2}
    			\node at  (\x,\y/5) [label={[xshift=.45cm, yshift=-0.38cm]\tiny\(\y\)}] {};
    	}
		\node at (2,0) [label={[yshift=-1.52cm]\(2\)}] {};
    	\foreach \y in {-1,0,1} \draw (-2.2,\y/5) node[crosspt]{};
    	\foreach \y in {-2,...,1} \draw (-0.2,\y/5) node[crosspt]{};
		\foreach \y in {-2,-1,1,2} \draw (-1.2,\y/5) node[crosspt]{};
    	\foreach \y in {-2,1,2} \draw (0.8,\y/5) node[crosspt]{};
    	\foreach\x/\y in {-2/-2, -2/2, -1/0, 0/2, 1/0,1/-1} 
    		\draw (\x-0.2,\y/5) node{\tiny {\(>0.5\)}};
    	\draw (2,0) node[crossgr]{};
    	\end{tikzpicture}
    \caption{}\label{fig:drittemgl4.2}
	\end{subfigure}
\caption{}
\end{figure}

	In order to make further progress, we study the entire function \(f\) and not 
	just its restriction to \(R_1\). The entries in \(R_0\) of Figure~\ref{fig:4lemma.1}
	are justified by~\ref{it:owl1},~\ref{it:owl4}, and Claim~\ref{clm:1420}. If \(Q, Q'\in R_1\) 
	are among the six boxes with \(f(Q), f(Q')>0.5\), then~\ref{it:owl3} 
	yields \(Q+Q'\not\in I\). This leads to the crosses in \(R_2\) of 
	Figure~\ref{fig:4lemma.1}. The crosses in \(R_{-1}\) are found similarly, 
	using \(R_{-1}=R_0-R_1\).  
	
\begin{figure}[h]    
    \centering
    \begin{tikzpicture}[scale=1.5]
    \foreach \x/\y/\z in {-2/2/2, -1/2/1}
			\filldraw[yellow] (\x,\y+\z/5) +(-.5,-.1) rectangle ++(.1,.1);
	\foreach \x/\y/\z in {-2/2/-2, -1/2/-1}
			\filldraw[orange] (\x,\y+\z/5) +(-.5,-.1) rectangle ++(.1,.1);
	\foreach \x/\y/\z in {0/2/2, 1/2/1}
			\filldraw[red] (\x,\y+\z/5) +(-.5,-.1) rectangle ++(.1,.1);
	\foreach \x/\y/\z in {0/2/-2, 1/2/-1}
			\filldraw[violet] (\x,\y+\z/5) +(-.5,-.1) rectangle ++(.1,.1);
    \foreach \x/\y in {-1/2, 0/1, 1/0, 2/-1}
			\filldraw[SeaGreen] (\x,\y) +(-.5,-.1) rectangle ++(.1,.1);
    \foreach \x in {-2,...,3} 
    	\draw[very thick] (-2.5, \x-.5)--(2.5, \x-.5) (\x-.5, -2.5)--(\x-.5, 2.5);
	\foreach \x in {-2,...,2}{
    	\node at (\x,0) [label={[yshift=-4.6cm] \(\x\)}] {};
        \node at  (0,\x) [label={[xshift=4.2cm, yshift=-0.4cm]\(\x\)}] {};
    }
    \foreach \x/\y in {-2/2,-1/2,0/2,1/2,-2/1,-1/1,0/1,1/1,-1/0,1/0,-1/-1,0/-1,1/-1,2/-1}{ 
    		\draw (\x+.1,\y-.5)--(\x+.1,\y+.5);
			\foreach \z in {1,...,4} \draw (\x-.5,\y+\z/5-.5)--(\x+.1, \y+\z/5-.5);
			\foreach \z in {-2,...,2}
    			\node at  (\x,\y+\z/5) [label={[xshift=.45cm, yshift=-0.38cm]\tiny\(\z\)}] {};
    	}
    \foreach \x/\y in {2/1, 0/0, 2/0, -2/-1, -2/0, 2/2} \draw (\x,\y) node[crossgr]{};
    \foreach \x/\y/\z in {-1/0/-1, -1/0/0, -1/0/1, 1/0/-1, 1/0/0, 1/0/1, -2/1/0, -1/1/-2,
    	-1/1/2, 0/1/0, 1/1/-2, 1/1/2}
		\draw (\x-0.2,\y+\z/5) node{\tiny {\(>0.5\)}};
	\foreach \x/\y in {-2/2,0/2,-1/1,1/1,-1/-1,1/-1} \foreach \z in {-1,0,1}
    	\draw (\x-0.2,\y+\z/5) node[crosspt]{};
	\foreach \x/\y in {-2/1,0/1,0/-1,2/-1} \foreach \z in {-2,-1,1,2}
    	\draw (\x-0.2,\y+\z/5) node[crosspt]{};
	\foreach \x in {-1,1} \foreach \y in {0,2} \foreach \z in {-2,2}
		\draw (\x-0.2,\y+\z/5) node[crosspt]{};
	\draw (0.8,2) node[crosspt]{};	
    \end{tikzpicture}
    \caption{}\label{fig:4lemma.1}
\end{figure}

	Next, the four green boxes form a subset \(B\) of the 
	line \((0, 1, 0)+\langle(1, -1, 0)\rangle\), whence 
	\[
		f(B)\le 3\,.
	\]
	Due to~\eqref{eq:6} and Figure~\ref{fig:4lemma.1} we have 
	\[
		f(R_{-2}\sm B)\le 6\,, \quad  f(R_{-1}\sm B)\le 5\,, \quad f(R_0\sm B)\le 5\,, 
		\quad \text{ and } \quad
		f(R_1\sm B)\le 5\,.
	\]
	As indicated by different colours, the eight boxes in \(R_2\sm B\) without crosses 
	can be grouped into four pairs whose difference is of the form \((\pm 1, 0, \pm 1)\). 
	So Claim~\ref{clm:1958} implies  
	\begin{equation}\label{eq:paare}
		\begin{cases}
		 	f(-2,2,2)+f(-1,2,1) &\le 1\,, \cr
		 	f(-2,2,-2)+f(-1,2,-1)&\le 1\,, \cr 
		 	f(0,2,2)+f(1,2,1)&\le 1\,, \cr
		 	f(0,2,-2)+f(1,2,-1)&\le 1\,. 
		 \end{cases}
	\end{equation}
	The addition of these four estimates yields \(f(R_2\sm B)\le 4\) and altogether 
	we obtain
	\[
		28\le \|f\|_1\le f(B)+\sum_{j\in \FF_5} f(R_j\sm B)\le 3+6+5+5+5+4=28\,,
	\]
	which means that equality holds throughout. 
	
	In particular, this shows \(f(P)=1\) 
	for the fifteen boxes \(P\in (R_1\cup R_0\cup R_{-1})\sm B\) which are known to 
	satisfy \(f(P)>0\), and that we have \(f(R_{-2})= 6\). 
	Using \(R_{-2}=R_{-1}+R_{-1}=R_{-1}-R_1\) and~\ref{it:owl3} 
	we can now show that, as indicated in Figure~\ref{fig:W}, many boxes are not 
	in \(R_2\cap I\). 
	
\begin{figure}[h]    
    \centering
    \begin{tikzpicture}[scale=1.5]
    \foreach \x/\y/\z in {-1/-2/1, 0/-2/2}
			\filldraw[yellow] (\x,\y+\z/5) +(-.5,-.1) rectangle ++(.1,.1);
	\foreach \x/\y/\z in {-1/-2/-1, 0/-2/-2}
			\filldraw[orange] (\x,\y+\z/5) +(-.5,-.1) rectangle ++(.1,.1);
	\foreach \x/\y/\z in {1/-2/1, 2/-2/2}
			\filldraw[red] (\x,\y+\z/5) +(-.5,-.1) rectangle ++(.1,.1);
	\foreach \x/\y/\z in {1/-2/-1, 2/-2/-2}
			\filldraw[violet] (\x,\y+\z/5) +(-.5,-.1) rectangle ++(.1,.1);
    \foreach \x/\y in {-1/2, 0/1, 1/0, 2/-1}
    \foreach \x/\y in {-2/1, -1/0, 0/-1, 1/-2}
			\filldraw[Cerulean] (\x,\y) +(-.5,-.1) rectangle ++(.1,.1);
    \foreach \x in {-2,...,3} 
    	\draw[very thick] (-2.5, \x-.5)--(2.5, \x-.5) (\x-.5, -2.5)--(\x-.5, 2.5);
	\foreach \x in {-2,...,2}{
    	\node at (\x,0) [label={[yshift=-4.6cm] \(\x\)}] {};
        \node at  (0,\x) [label={[xshift=4.2cm, yshift=-0.4cm]\(\x\)}] {};
    }
    \foreach \x/\y in {-2/2,-1/2,0/2,1/2,-2/1,-1/1,0/1,1/1,-1/0,1/0,-1/-1,0/-1,1/-1,2/-1,
    	-1/-2,0/-2,1/-2,2/-2}{ 
    		\draw (\x+.1,\y-.5)--(\x+.1,\y+.5);
			\foreach \z in {1,...,4} \draw (\x-.5,\y+\z/5-.5)--(\x+.1, \y+\z/5-.5);
			\foreach \z in {-2,...,2}
    			\node at  (\x,\y+\z/5) [label={[xshift=.45cm, yshift=-0.38cm]\tiny\(\z\)}] {};
    	}
    \foreach \x/\y in {2/1, 0/0, 2/0, -2/-1, -2/0, 2/2, -2/-2} \draw (\x,\y) node[crossgr]{};
    \foreach \x/\y/\z in {-1/0/-1, -1/0/0, -1/0/1, 1/0/-1, 1/0/1, -2/1/0, -1/1/-2,
    	-1/1/2, 1/1/-2, 1/1/2, -1/-1/2, -1/-1/-2, 0/-1/0, 1/-1/2, 1/-1/-2}
		\draw (\x-0.2,\y+\z/5) node{\tiny {\(1\)}};
	\foreach \x/\y in {-2/2,0/2,-1/1,1/1,-1/-1,1/-1,0/-2,2/-2} \foreach \z in {-1,0,1}
    	\draw (\x-0.2,\y+\z/5) node[crosspt]{};
	\foreach \x/\y in {-2/1,0/1,0/-1,2/-1} \foreach \z in {-2,-1,1,2}
    	\draw (\x-0.2,\y+\z/5) node[crosspt]{};
	\foreach \x in {-1,1} \foreach \y in {-2,0,2} \foreach \z in {-2,2}
		\draw (\x-0.2,\y+\z/5) node[crosspt]{};
	\draw (0.8,2) node[crosspt]{};	
	\foreach \x/\y/\z in {0/1/0,1/0/0} \draw (\x-0.2,\y+\z/5) node{\tiny {\(>0.5\)}};
    \end{tikzpicture}
    \caption{}\label{fig:W}
\end{figure}

	There remain ten boxes in \(R_2\) without crosses. Similar to~\eqref{eq:paare}
	we can form four pairs of boxes differing by \((\pm 1, 0, \pm 1)\). This yields 
	\(6=f(R_{-2})\le f(-1, -2, 0)+f(1, -2, 0)+4\), wherefore~\({f(1, -2, 0)=1}\).
	We now obtain the contradiction that the four blue boxes form a subset \(C\) 
	of the line \((0, -1, 0)+\langle(1, -1, 0)\rangle\) with \(f(C)=4\). 
\end{proof}

Let us summarise what we know about the ``rich'' rows \(R_y\) whose index \(y\) 
is in the set
\[
	\Psi=\bigl\{y\in \FF_5^\times\colon g(R_y)>5.5\bigr\}\,.
\]
If \(y\in\Psi\), then Claims~\ref{5.5_5}\,--\,\ref{5.5_4} 
imply \(|R_y\cap J|=3\). Among the three boxes 
in~\({R_y\cap J}\) there need to be two horizontally adjacent ones, i.e., there is 
some \(x\in \FF_5\)
such that \((x, y)\) and~\({(x+1, y)}\) are in \(J\). By~\eqref{eq:nachbarnkleinerfall} we have 
\(g(x, y)+g(x+1, y)\le 3\) and thus the remaining box~\({(x', y)}\) in \(R_j\cap J\) 
satisfies~\({g(x', y)>2.5}\). Combined with~\(g(1, 0)>2.5\) this yields 
\((x'\pm 1, y)\not\in J\), whence~\({x'=x+2}\). 
		
\begin{claim}\label{clm:psi}
	If \(y\in\Psi\), then \(2y\not\in\Psi\). 
\end{claim}

\begin{proof}
	Assume contrariwise that \(1, 2\in \Psi\). By \(1\in\Psi\) and the above discussion, 
	we can suppose without loss of generality that \(g(0, 1)>2.5\). As indicated in 
	Figure~\ref{fig:psia},
	the fishiness of~\(g\) implies that sums and differences of \((1, 0)\), \((-1, 0)\),
	and \((0, 1)\) cannot be in \(J\). The boxes~\({(\pm 2, 1)}\), on the other hand, 
	are in \(J\). 
	Because of \((\pm 2, 2)-(0,1)=(\pm 2, 1)\in J\) we need to have \(g(\pm 2, 2)<2.5\). 
	Thus the box \((x, 2)\in R_2\) satisfying \(g(x, 2)>2.5\) can only be \((-1, 2)\) 
	or~\({(1, 2)}\), and without loss of generality we can suppose \(g(1, 2)>2.5\) 
	(see Figure~\ref{fig:psib}). 
	
\begin{figure}[h]
	\centering
\hspace{3em}
	\begin{subfigure}[b]{.45\textwidth}
		\centering
		\begin{tikzpicture}[scale=.9]
		\foreach \x in {-2,...,3} \draw (-2.5, \x-.5)--(2.5, \x-.5) (\x-.5, -2.5)--(\x-.5, 2.5);
		\foreach \x/\y in {0/2,-1/1,1/1,-2/0,0/0,2/0,-1/-1,1/-1} \draw (\x,\y) node[cross]{};
		\foreach \x/\y in {-1/0,0/1,1/0} \draw (\x,\y) node[font=\footnotesize]{\(>2.5\)}; 
		\foreach \x/\y in {-2/2,2/2} \draw (\x,\y) node[font=\footnotesize]{\(<2.5\)}; 
		\foreach \x in {-2,2} \draw (\x, 1) node[font=\footnotesize]{\(>0\)};
		\end{tikzpicture}
		\caption{}\label{fig:psia}
	\end{subfigure}
\hfill
	\begin{subfigure}[b]{.45\textwidth}
		\centering
		\begin{tikzpicture}[scale=.9]
		\foreach \x/\y in {-2/1, -2/-1, 1/0}
			\filldraw[Cerulean] (\x,\y) +(-.5,-.5) rectangle ++(.5,.5);
		\foreach \x in {-2,...,3} \draw (-2.5, \x-.5)--(2.5, \x-.5) (\x-.5, -2.5)--(\x-.5, 2.5);
		\foreach \x/\y in {0/2,-1/1,1/1,-2/0,0/0,2/0,-1/-1,1/-1,2/-1}\draw (\x,\y) node[cross]{};
		\foreach \x/\y in {-1/0,0/1,1/0,1/2} \draw (\x,\y) node[font=\footnotesize]{\(>2.5\)};
		\draw (0,-1) node[font=\footnotesize]{\(<2.5\)};
		\draw (2,1) node[font=\footnotesize]{\(\le 2\)};
		\end{tikzpicture}
		\caption{}\label{fig:psib}
	\end{subfigure}
	\caption{}
\hspace{3em}
\end{figure}
    
	We have now found four boxes \(Q\) satisfying \(g(Q)>2.5\) and by Lemma~\ref{lem:P4P5} 
	there are no further such boxes. In particular, this proves \(g(0, -1)<2.5\) and 
	\(-2\not\in\Psi\), i.e.,
	\[
		g(R_{-2})<5.5\,.
	\]
	Moreover, \((2, -1)=2\cdot (1, 2)\) entails \(g(2, -1)=0\). Next, we contend
	\[
		g(-2, 1)+g(-2, -1)+g(1, 0)\le 6\,.
	\]
	Due to \(\|g\|_\infty\le 3\) this is trivial if one of the summands on the left side 
	vanishes. In the remaining case, we appeal to \((-2, 1)+(-2, -1)=(1, 0)\) and the 
	fact that \(g\) is fishy. 
	
	Finally, \(g(-2, 1)+g(2, 1)=g(R_1)-g(0, 1)>5.5-3>2.2\) and~\ref{it:owl6} 
	imply 
	\[
		g(2, 1)\le |I_{2, 1}|\le 2\,.
	\]
	Adding everything, we arrive first at
	\[
		g(R_{-1})+g(R_0)+g(R_1)<6+3+3+2+2.5=16.5\,,
	\]
	and then at the contradiction 
	\[
		28\le \|g\|_1<16.5+6+5.5=28\,. \qedhere
	\]
\end{proof}

\begin{claim}\label{clm:strange}
	There exists some \(y\in\Psi\) such that \(g(R_{2y})\ge 5\).
\end{claim}

\begin{proof}
	If \(\Psi\) is empty, we get the contradiction \(28\le \|g\|_1<6+4\cdot 5.5=28\). 
	So there is some~\({y\in\Psi}\) and, since we are otherwise done, 
	we can assume \(g(R_{2y})<5\). By Claim~\ref{clm:psi} applied to~\(-2y\) instead 
	of~\(y\) we obtain \((-2y)\not\in \Psi\). But now 
	\[
		g(R_{-y})
		= 
		\|g\|_1-g(R_0)-g(R_{-2y})-g(R_y)-g(R_{2y})
		>
		28-6-5.5-6-5=5.5
	\]
	shows \(-y\in \Psi\) and due to 
	\[
		g(R_{-2y})
		=
		\|g\|_1-g(R_0)-g(R_{-y})-g(R_y)-g(R_{2y})
		>
		28-6-6-6-5=5
	\]
	the desired property is possessed by \(-y\).
\end{proof}

We proceed with some assumptions that can be made without loss of generality. 
Owing to Claim~\ref{clm:strange} we can suppose 
\begin{equation}\label{eq:1301}
	g(R_1)>5.5
	\quad \text{ and } \quad 
	g(R_2)\ge 5\,.
\end{equation}
By Claim~\ref{clm:psi} applied to \(y=1\) and \(y=-2\) we also have \(\pm 2\not\in \Psi\), i.e., 
\begin{equation}\label{eq:0015}
	\max\{g(R_2), g(R_{-2})\}<5.5\,.
\end{equation}
Because of \(\|g\|_\infty \ge 28\) this implies 
\begin{equation}\label{eq:17}
	g(R_{-1})+g(R_0)+g(R_1)>17\,.
\end{equation}
Due to \(1\in\Psi\) we can suppose \(g(0, 1)>2.5\), which immediately justifies 
all the big crosses in Figure~\ref{fig:grossesblid5} except for the one in \((0, -2)\). 
In order to see \((0, -2)\not\in J\) we subtract from~\eqref{eq:17} the upper bounds 
\begin{align*}
	g(R_0)&\le 6\,, \\
	g(-2, 1)+g(2, 1) &\le 3\,, \\
	g(-2, -1)+g(2, -1)&\le 3\,, 
\end{align*}
the latter two of which follow from~\eqref{eq:nachbarnkleinerfall}. 
This yields \(g(0, 1)+g(0, -1)>5\) and because of~\({(0, -1)-(0, 1)=(0,-2)}\) 
we have indeed \((0, -2)\not\in J\). 

Having thereby explained the big crosses in Figure~\ref{fig:grossesblid5} we take a closer 
look at the row~\(R_1\). Due to \(g(-2, 1)+g(2, 1)=g(R_1)-g(0, 1)>5.5-3=2.5\) 
Claim~\ref{clm:2001} allows us to suppose, without loss of generality, that the internal 
structure of the boxes \((-2, 1)\) and \((2, 1)\) is as suggested by 
Figure~\ref{fig:grossesblid5}. The seven small crosses in~\(R_{-1}\) are then deduced 
from~\ref{it:owl3}, using~\({L_{2, -1}=L_{-1, 0}-L_{2, 1}}\) as well 
as \(L_{-2, -1}=L_{1, 0}-L_{-2, 1}\). 

Notice that~\eqref{eq:17} yields \(g(-2, -1)+g(2,-1)>17-6-6-3=2\), 
whence \((-2, -1, 2)\), \((-2, -1, -2)\), and \((2, -1, 0)\) are in \(I\). 
Similarly, if \(Q\), \(Q'\) are two distinct ones among the six boxes of the form
\((\pm 2, \pm 1, z)\) belonging to~\(I\), then~\eqref{eq:17} yields 
\(f(Q)+f(Q')>17-4\cdot 1-4\cdot 3=1\), whence \(Q\pm Q'\not\in I\). This argument 
leads to the remaining small crosses in Figure~\ref{fig:grossesblid5}. 
The meaning of the colours is soon going to become clear. 
  
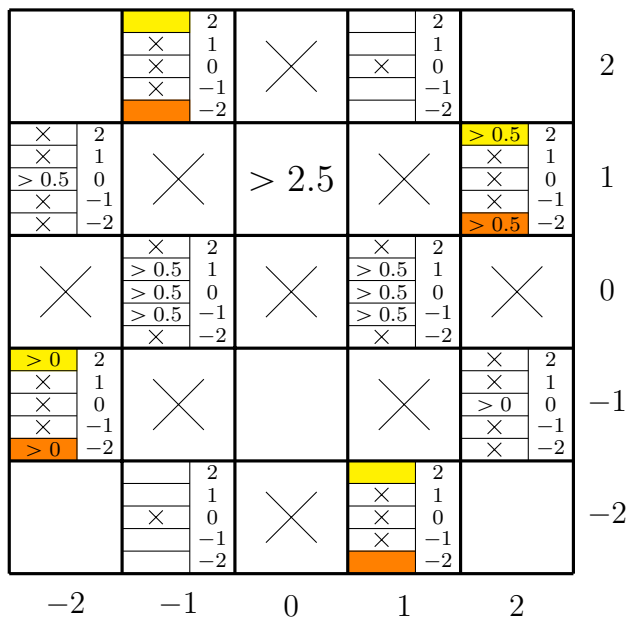
\begin{figure}[h]
    \centering
    \begin{tikzpicture}[scale=1.5]
    \foreach \x/\y in {-1/2, 2/1, 1/-2, -2/-1}
			\filldraw[orange] (\x,\y-.4) +(-.5,-.1) rectangle ++(.1,.1);
	\foreach \x/\y in {-1/2, 2/1, 1/-2, -2/-1}
			\filldraw[yellow] (\x,\y+.4) +(-.5,-.1) rectangle ++(.1,.1);
    \foreach \x in {-2,...,3} 
    	\draw[very thick] (-2.5, \x-.5)--(2.5, \x-.5) (\x-.5, -2.5)--(\x-.5, 2.5);
	\foreach \x in {-2,...,2}{
    	\node at (\x,0) [label={[yshift=-4.6cm] \(\x\)}] {};
        \node at  (0,\x) [label={[xshift=4.2cm, yshift=-0.4cm]\(\x\)}] {};
    }
    \foreach \x/\y in {-1/2,1/2,-2/1,2/1,-1/0,1/0,-2/-1,2/-1,-1/-2,1/-2}{ 
    		\draw (\x+.1,\y-.5)--(\x+.1,\y+.5);
			\foreach \z in {1,...,4} \draw (\x-.5,\y+\z/5-.5)--(\x+.1, \y+\z/5-.5);
			\foreach \z in {-2,...,2}
    			\node at  (\x,\y+\z/5) [label={[xshift=.45cm, yshift=-0.38cm]\tiny\(\z\)}] {};
    	}
    \foreach \x/\y in {0/2,-1/1,1/1,-2/0,0/0,2/0,-1/-1,1/-1,0/-2} \draw (\x,\y) node[crossgr]{};
    \foreach \x/\y/\z in {-2/1/0,2/1/2,2/1/-2,-1/0/-1,-1/0/0,-1/0/1,1/0/-1,1/0/0,1/0/1}
		\draw (\x-0.2,\y+\z/5) node{\tiny {\(>0.5\)}};
	\foreach \x/\y in {-1/2,2/1,-2/-1,1/-2} \foreach \z in {-1,0,1}
    	\draw (\x-0.2,\y+\z/5) node[crosspt]{};
	\foreach \x/\y in {-2/1,2/-1} \foreach \z in {-2,-1,1,2}
    	\draw (\x-0.2,\y+\z/5) node[crosspt]{};
	\foreach \x in {-1,1} \foreach \z in {-2,2}
		\draw (\x-0.2,\z/5) node[crosspt]{};
	\draw (0.8,2) node[crosspt]{};
	\draw (-1.2,-2) node[crosspt]{};
	\draw (0,1) node{\large \(>2.5\)};
	\foreach \x/\y/\z in {-2/-1/2,-2/-1/-2,2/-1/0}
		\draw (\x-0.2,\y+\z/5) node{\tiny {\(>0\)}};
	\end{tikzpicture}
    \caption{}\label{fig:grossesblid5}
\end{figure}

Roughly speaking, our strategy for concluding the argument is to estimate~\(g\)
on boxes in~\(R_2\) until we reach a contradiction. 

\begin{claim}\label{clm:1129}
	We have \(g(-1, 2)<1.5\).
\end{claim}

\begin{proof}
	In Figure~\ref{fig:grossesblid5} the yellow and orange boxes form subsets \(A\) and \(B\)
	of the lines 
   	\[
		(0, 0, 2)+\langle (1, -2, 0)\rangle
		\quad \text{ and } \quad 
		(0, 0, -2)+\langle (1, -2, 0)\rangle\,,
	\]
	respectively. This proves \(f(A)+f(B)\le 6\), which together with 
	\[
		\sum_{y\in\FF_5\sm\{2\}} f\bigl(R_y\sm (A\cup B)\bigr)
		\le 3+g(2, -2)+4+6+4=17+g(2, -2)
	\]
	yields 
	\[
		g(R_2)-g(-1, 2)\ge 28-6-\bigl(17+g(2, -2)\bigr)=5-g(2, -2)\,.
	\]
	So if contrary to our claim \(g(-1, 2)>1.5\) holds,~\eqref{eq:0015} 
	entails 
	\[
		g(2, -2)\ge 5+g(-1, 2)-g(R_2)>5+1.5-5.5=1\,,
	\]
	whence \(g(1, 0)+g(-1, 2)+\lceil g(2, -2)\rceil>2.5+1.5+2=6\). 
	But due to \((1, 0)-(-1, 2)=(2, -2)\) this contradicts~\ref{it:fish3}.	
\end{proof}

Next we study the `right half' of \(R_2\). 

\begin{claim}\label{clm:ersterclaim}
    We have \(g(1,2)+g(2,2)<2.5\).
\end{claim}

\begin{proof}
    By subtracting \(g(R_{-1})+g(R_0)+g(-2,1)\le 6+6+1=13\) from~\eqref{eq:17} 
    we obtain~\({g(0,1)+g(2,1)>4}\). Since \(g\) is fishy, this implies \(g(2,2)\le 1\). 
    Assuming from now on that the claim fails,~\ref{it:owl6} tells us~\(|I_{2,2}|=1\) 
    and~\(|I_{1,2}|=2\).
    
\begin{figure}[h]
    \centering
    \begin{tikzpicture}[scale=1.5]
    \foreach \x in {-2,...,3} 
    	\draw[very thick] (-2.5, \x-.5)--(2.5, \x-.5) (\x-.5, -2.5)--(\x-.5, 2.5);
	\foreach \x in {-2,...,2}{
    	\node at (\x,0) [label={[yshift=-4.6cm] \(\x\)}] {};
        \node at  (0,\x) [label={[xshift=4.2cm, yshift=-0.4cm]\(\x\)}] {};
    }
    \foreach \x/\y in {-1/2,1/2,-2/1,2/1,-1/0,1/0,-2/-1,2/-1,-1/-2,1/-2,2/2,0/1}{ 
    		\draw (\x+.1,\y-.5)--(\x+.1,\y+.5);
			\foreach \z in {1,...,4} \draw (\x-.5,\y+\z/5-.5)--(\x+.1, \y+\z/5-.5);
			\foreach \z in {-2,...,2}
    			\node at  (\x,\y+\z/5) [label={[xshift=.45cm, yshift=-0.38cm]\tiny\(\z\)}] {};
    	}
    \foreach \x/\y in {0/2,-1/1,1/1,-2/0,0/0,2/0,-1/-1,1/-1,0/-2,2/-2} 
    	\draw (\x,\y) node[crossgr]{};
    \foreach \x/\y/\z in {-2/1/0,2/1/2,2/1/-2,-1/0/-1,-1/0/0,-1/0/1,1/0/-1,1/0/0,1/0/1,
    	0/1/-2,0/1/-1,0/1/0,1/2/1,1/2/2,2/2/-1}
		\draw (\x-0.2,\y+\z/5) node{\tiny {\(>0.5\)}};
	\foreach \x/\y in {-1/2,2/1,-2/-1,1/-2} \foreach \z in {-1,0,1}
    	\draw (\x-0.2,\y+\z/5) node[crosspt]{};
	\foreach \x/\y in {-2/1,2/-1} \foreach \z in {-2,-1,1,2}
    	\draw (\x-0.2,\y+\z/5) node[crosspt]{};
	\foreach \x in {-1,1} \foreach \z in {-2,2}
		\draw (\x-0.2,\z/5) node[crosspt]{};
	\foreach \x/\y/\z in {1/2/-2,1/2/-1,1/2/0,2/2/-2,2/2/0,2/2/1,2/2/2,0/1/1,
		0/1/2,-1/-2/0,1/-2/2}
		\draw (\x-0.2,\y+\z/5) node[crosspt]{};
	\end{tikzpicture}
    \caption{}\label{fig:grossesblidersterclaim}
\end{figure}

	Because of~\({0\not\in I_{1, 2}}\) Claim~\ref{clm:2001} 
    implies that the unique element of \(I_{2,2}\) cannot be~\(+2\) or~\({-2}\). 
    Suppose next that~\({I_{2,2}=\{0\}}\), which
    by Claim~\ref{clm:2001} implies \(I_{1, 2}=\{-2, 2\}\). 
    Together with 
    \[
    	(1, 2, 2)+(1, 2, -2)=(2, -1, 0)\in I
	\]
	and~\ref{it:owl3} this leads to the contradiction 
    \[
    	g(1,2)=f(1, 2, 2)+f(1, 2, -2)\le 1\,.
	\]
    
    Without loss of generality it remains to discuss the case \(I_{2,2}=\{-1\}\),
    whence \(I_{1, 2}=\{1,2\}\). By the failure of our claim the three  
    boxes \(P\in (L_{1,2}\cup L_{2,2})\cap I\) satisfy \(f(P)>0.5\)
    (see Figure~\ref{fig:grossesblidersterclaim}). Using \(L_{2,2}-L_{2,1}=L_{0,1}\) 
    and~\ref{it:owl3} we infer \(I_{0,1}\subseteq \{-2,-1,0\}\), 
    which together with \(g(0,1)>2.5\) entails \(f(0, 1, z)>0.5\) for all \(z\in\{-2, -1, 0\}\).
    Now \((2, -2)\not\in J\) follows from~\({L_{2, -2}=L_{2,2}+L_{0,1}=L_{-1,0}-L_{2,2}}\) 
    and~\ref{it:owl3}. 
    Moreover, \((1, -2, 2)=(0,1,0)+(1, 2, 2)\) cannot be in \(I\). This explains the 
    new crosses in \(R_{-2}\) we drew in Figure~\ref{fig:grossesblidersterclaim}.

	Because of~\eqref{eq:nachbarnkleinerfall} they lead to \(g(R_{-2})\le 4\),
	which combined with~\eqref{eq:6} and~\eqref{eq:0015} gives the contradiction
	\[
    28\le \|g\|_1<4+6+6+6+5.5=27.5\,. \qedhere
    \]
\end{proof}

We are now ready to complete the proof of Proposition~\ref{prop:0135}. 
By~\eqref{eq:1301} and Claim~\ref{clm:ersterclaim} we have \(g(-2,2)+g(-1,2)>2.5\).
Together with Claim~\ref{clm:1129} this proves \(g(-2, 2)>1\). We also have \(g(-1, 2)>0\),
because \(g(-2, 2)>2.5\) and \(g(0, 1)\) implied \(g(-2, 1)=0\), which would contradict 
\(f(-2, 1, 0)>0\).

\begin{figure}[h]
    \centering
    \begin{tikzpicture}[scale=1.5]
    \foreach \x/\y/\z in {-1/2/2, 0/1/-2, 1/0/-1, 2/-1/0}
			\filldraw[SeaGreen] (\x,\y+\z/5) +(-.5,-.1) rectangle ++(.1,.1);
    \foreach \x in {-2,...,3} 
    	\draw[very thick] (-2.5, \x-.5)--(2.5, \x-.5) (\x-.5, -2.5)--(\x-.5, 2.5);
	\foreach \x in {-2,...,2}{
    	\node at (\x,0) [label={[yshift=-4.6cm] \(\x\)}] {};
        \node at  (0,\x) [label={[xshift=4.2cm, yshift=-0.4cm]\(\x\)}] {};
    }
    \foreach \x/\y in {-1/2,1/2,-2/1,2/1,-1/0,1/0,-2/-1,2/-1,-1/-2,1/-2,0/1,-2/2}{ 
    		\draw (\x+.1,\y-.5)--(\x+.1,\y+.5);
			\foreach \z in {1,...,4} \draw (\x-.5,\y+\z/5-.5)--(\x+.1, \y+\z/5-.5);
			\foreach \z in {-2,...,2}
    			\node at  (\x,\y+\z/5) [label={[xshift=.45cm, yshift=-0.38cm]\tiny\(\z\)}] {};
    	}
    \foreach \x/\y in {0/2,-1/1,1/1,-2/0,0/0,2/0,-1/-1,1/-1,0/-2} \draw (\x,\y) node[crossgr]{};
    \foreach \x/\y/\z in {-2/1/0,2/1/2,2/1/-2,-1/0/-1,-1/0/0,-1/0/1,1/0/-1,1/0/0,1/0/1,
    	0/1/-2,0/1/1,0/1/2,-2/2/-1,-2/2/0,-1/2/2}
		\draw (\x-0.2,\y+\z/5) node{\tiny {\(>0.5\)}};
	\foreach \x/\y in {-1/2,2/1,-2/-1,1/-2} \foreach \z in {-1,0,1}
    	\draw (\x-0.2,\y+\z/5) node[crosspt]{};
	\foreach \x/\y in {-2/1,2/-1} \foreach \z in {-2,-1,1,2}
    	\draw (\x-0.2,\y+\z/5) node[crosspt]{};
	\foreach \x in {-1,1} \foreach \z in {-2,2}
		\draw (\x-0.2,\z/5) node[crosspt]{};
	\foreach \x/\y/\z in {-2/2/-2,-2/2/1,-2/2/2,-1/2/-2,0/1/-1,0/1/0,1/2/0,-1/-2/0}
		\draw (\x-0.2,\y+\z/5) node[crosspt]{};
	\end{tikzpicture}
    \caption{}\label{fig:grossesblidzweiterclaim1}
\end{figure}

Now~\ref{it:owl6} tells us \(|I_{-2,2}|=2\) and \(|I_{-1,2}|=1\). 
By Claim~\ref{clm:2001} we can suppose, without loss of generality, that 
\(I_{-1,2}=\{2\}\) and \(I_{-2,2}=\{-1, 0\}\). Using \(L_{0,1}=L_{-2,2}-L_{-2,1}\)
we find~\({L_{0,1}=\{-2, 1, 2\}}\). The current situation is drawn in 
Figure~\ref{fig:grossesblidzweiterclaim1}. Its four green boxes form a subset \(D\)
of the line \((0, 1, 2)+\langle(1, -1, 1)\rangle\). So we have \(f(D)\le 3\). It is plain 
that \(f(R_y\sm D)\le 5\) holds for \(y\in \{-1, 0, 1\}\). 
Moreover, Claim~\ref{clm:ersterclaim} and~\eqref{eq:0015} 
show \(f(R_2\sm D)<2+2.5=4.5\) and \(f(R_{-2}\sm D)<5.5\), respectively.
So altogether the contradiction 
\[
    28\le f(D)+\sum_{y\in\FF_5} f(R_y\sm D)<3+5.5+3\cdot 5+4.5=28
\]
arises and Proposition~\ref{prop:0135} is proved. 

\section{Applications of Kneser's theorem}\label{sec:kneser}

\subsection{Hyperplanes}
In this subsection we study dense sum-free sets contained in the union of a small number 
of parallel hyperplanes. 

\begin{lemma}\label{lem:2planes}
	If for some \(n\ge 2\) a sum-free set \(A\subseteq \FF_5^n\) of size \(|A|>5^{n-1}\)
	can be covered by two parallel hyperplanes, then it is normal. 
\end{lemma}

%
%

\begin{proof}
	Suppose that for some hyperplane \(H\le \FF_5^n\), vector \(v\in \FF_5^n\sm H\), 
	and scalars \(\lambda, \mu\in \FF_5\) we have 
	\[
		A\subseteq (\lambda v+H)\cup (\mu v+H)\,.
	\]
	Due to \(|A|>|H|\) the sets \(A_\lambda=A\cap (\lambda v+H)\) and \(A_\mu=A\cap (\mu v+H)\)
	are nonempty. The group~\(\FF_5^\times\) acts by multiplication on the two-element subsets 
	of~\(\FF_5\), with three orbits represented by \(\{0, 1\}\), \(\{1, 2\}\), and \(\{1, -1\}\).
	We can therefore suppose that \(\{\lambda, \mu\}\) is one of these three sets. 
	In the last case,~\(A\) is clearly normal, so it suffices to show that the first two cases
	lead to contradictions. 
	
	Suppose first that \(\lambda=0\) and \(\mu=1\). Since \(A\) is sum-free, 
	we have \({(A_0+A_1)\cap A_1=\vn}\). But now 
	\(5^{n-1}=|H+v|\ge |A_0+A_1|+|A_1|\ge |A_0|+|A_1|=|A|\) contradicts our assumption 
	on the size of \(A\).

	The case \(\lambda=1\) and \(\mu=2\) is similar, the contradiction being 
	\[
		5^{n-1}=|H+2v|\ge |A_1+A_1|+|A_2|\ge |A_1|+|A_2|=|A|\,. \qedhere
	\]
\end{proof}

Our remaining applications of Kneser's theorem factorise through the following result. 

\begin{lemma}\label{lem:1419}
	Let \(X, Y, Z\subseteq \FF_5^\ell\) be three nonempty sets 
	satisfying \((X+Y)\cap Z=\vn\). If there exists an integer \(t\)
	such that \(|X|+|Y|>(25-t)5^{\ell-2}\) and \(|Z|>t\cdot 5^{\ell-2}\), 
	then \(t\) is not divisible by \(5\) and there exists a hyperplane \(K\le \FF_5^\ell\) 
	together with positive integers \(a\), \(b\) 
	such that \(a+b+\lceil t/5\rceil=6\), and \(X\), \(Y\), \(Z\) can be covered by \(a\), 
	\(b\), \(\lceil t/5\rceil\) translates of \(K\), respectively. 
\end{lemma}

\begin{proof}
	Since \(Z\ne\vn\) implies \(X+Y\ne \FF^\ell_5\), the symmetry set \(K=\Sym(X+Y)\) 
	is a proper linear subspace of \(\FF^\ell_5\). Moreover, the 
	fact that the sum set \(X+Y\) is a union of translates of~\(K\) implies that it is
	even disjoint to \(Z+K\) and, therefore, Kneser's theorem reveals
	\begin{equation}\label{eq:1404}
		 5^\ell\ge |X+Y|+|Z+K|\ge |X+K|+|Y+K|+|Z+K|-|K|\,.
	\end{equation}
	  
	Now suppose for the sake of contradiction that \(\dim(K)\le \ell-2\), so that not 
	only \(|Z+K|\) but also \(t\cdot 5^{\ell-2}\) is divisible by \(|K|\). For this reason,
	\(|Z|>t\cdot 5^{\ell-2}\) entails \(|Z+K|-|K|\ge t\cdot 5^{\ell-2}\) and~\eqref{eq:1404}
	leads to the contradiction 
	\[
		(25-t)5^{\ell-2}\ge |X+K|+|Y+K|\ge |X|+|Y|\,.
	\]
	
	We have thereby shown that \(K\) is a hyperplane. Define the positive 
	integers \(a\), \(b\), and \(c\) by 
	\[
		|X+K|=a|K|\,, \quad 
		|Y+K|=b|K|\,, 
		\quad \text{ and } \quad 
		|Z+K|=c|K|\,.
	\]
	Clearly, \(X\), \(Y\), and \(Z\) can be covered by \(a\), \(b\), and \(c\) translates 
	of \(K\), respectively. Moreover,~\eqref{eq:1404} shows \(a+b+c\le 6\), while the 
	lower bounds 
	\[
		a+b\ge \frac{|X|+|Y|}{|K|}>\frac{25-t}{5}
	\]
	and 
	\[
		c\ge \frac{|Z|}{|K|}>\frac{t}{5}
	\]
	are clear. For these reasons, \(t\) cannot be divisible by \(5\) and the equalities  
	\(a+b+c=6\) as well as \(c=\lceil t/5\rceil\) hold. 
\end{proof}

\begin{cor}\label{cor:1135}
	If three nonempty sets \(X, Y, Z\subseteq \FF_5^\ell\) satisfy \((X+Y)\cap Z=\vn\) and 
	\[
		|X|+|Y|+|Z|>26\cdot 5^{\ell-2}\,,
	\]
	then there exists a linear epimorphism \(\phi\colon \FF^\ell_5\lra\FF_5\) 
	such that \(|\phi(X)|+|\phi(Y)|+|\phi(Z)|=6\).
\end{cor}

\begin{proof}
	Setting \(t=\lceil |Z|/5^{\ell-2}\rceil-1\) 
	we have 
	\[
		(t+1)5^{\ell-2}\ge |Z|>t\cdot 5^{\ell-2}
	\]
	and 
	\[
		|X|+|Y|>\bigl(26-(t+1)\bigr)5^{\ell-2}=(25-t)5^{\ell-2}\,.
	\] 
	
	Thus Lemma~\ref{lem:1419} delivers a hyperplane \(K\le \FF_5^\ell\) 
	and integers \(a\), \(b\), \(c\) such that \(a+b+c=6\) and~\(X\), \(Y\), \(Z\) 
	can be covered by \(a\), \(b\), \(c\) translates of \(K\), respectively. 
	Now any linear epi\-morphism \(\phi\colon \FF^\ell_5\lra\FF_5\) with 
	kernel \(K\) satisfies \(|\phi(X)|+|\phi(Y)|+|\phi(Z)|\le a+b+c=6\).
	Finally, \(5.2\cdot 5^{\ell-1}<|X|+|Y|+|Z|\le (|\phi(X)|+|\phi(Y)|+|\phi(Z)|)5^{\ell-1}\)
	shows that this needs to hold with equality. 
\end{proof}

Theorem~\ref{5n=2} yields an analogue of Lemma~\ref{lem:2planes} concerning sum-free 
sets that can be covered by six translates of a subspace with codimension \(2\). 

\begin{lemma}\label{lem:62}
	If \(A\subseteq \FF_5^n\) denotes a sum-free set of size \(|A|>5^{n-1}\) and there 
	exists a linear epimorphism \(\phi\colon \FF_5^n\lra \FF_5^2\) such that \(|\phi(A)|\le 6\),
	then \(A\) is normal. 
\end{lemma}

\begin{proof}
	Let us write \(A_P=\phi^{-1}(P)\cap A\) for every \(P\in \FF_5^2\). By averaging
	there exists a point~\({P\in \phi(A)}\) such that \(|A_P|>\frac 16|A|>0.5\cdot 5^{n-2}\).
	Now Corollary~\ref{Bdumm} yields \(A_P-A_P=\ker(\phi)\), whence \((0, 0)\not\in \phi(A)\). 
	Similarly, we have \(|A_P|+|A_Q|\ge |A|-4\cdot 5^{n-2}> 5^{n-2}\) for all 
	distinct points \(P, Q\in \phi(A)\), which leads to \(A_{P\pm Q}=\vn\). 
	For these reasons, \(\phi(A)\) is sum-free.
	
	Moreover,~\({|A|>5^{n-1}}\) implies \(\phi(A)\ge 6\) and by Theorem~\ref{5n=2} \(\phi(A)\)
	is normal. If~\({L\subseteq \FF_5^2}\) denotes any line with \(\phi(A)\subseteq L\cup (-L)\), 
	then the affine hyperplane \(H=\phi^{-1}(L)\) satisfies~\({A\subseteq H\cup (-H)}\).  
\end{proof}

Under a slightly more restrictive density assumption on \(A\), we can now extend Lemma~\ref{lem:2planes} to sets coverable by three parallel hyperplanes. 

\begin{lemma}\label{lem:3planes}
    If for some \(n\ge 2\) a sum-free set \(A\subseteq \FF_5^n\) of 
    size \(|A|\ge 28\cdot 5^{n-3}\) can be covered by three parallel 
    hyperplanes, then it is normal. 
\end{lemma}

\begin{proof}
	Again we suppose 
	\[
		A\subseteq (\lambda v+H)\cup (\mu v+H)\cup (\nu v+H)\,,
	\]
	where \(H\le \FF_5^n\) denotes a hyperplane, \(v\in\FF_5^n\sm H\), 
	and \(\lambda, \mu, \nu\in \FF_5\) are distinct. By Lemma~\ref{lem:2planes}
	it suffices to treat the case that 
	\[
		A_\lambda=A\cap (\lambda v+H)\,, \quad 
		A_\mu=A\cap (\mu v+H)\,, 
		\quad \text{ and } \quad 
		A_\nu=A\cap (\nu v+H)
	\]
	are nonempty. By multiplicative symmetry, we only need to consider the cases  
	that \(\{\lambda, \mu, \nu\}\) is one of~\(\{1, 2, -2\}\), \(\{0, 1, -1\}\), 
	or \(\{0, 1, 2\}\). 
	
	\medskip

	{\it \hskip2em First Case: Either \(\{\lambda, \mu, \nu\}=\{1, 2, -2\}\) 
		or \(\{\lambda, \mu, \nu\}=\{0, 1, -1\}\).}
        
	\smallskip
	
	Without loss of generality, we can now suppose \(\lambda=\mu+\nu\). 
	By Corollary~\ref{cor:1135} applied to~\(H\) instead of \(\FF_5^\ell\) 
	and to the sets \(X=A_\mu-\mu v\), \(Y=A_\nu-\nu v\), \(Z=A_\lambda-\lambda v\)
	we get a certain linear epimorphism \(\phi\colon H\lra \FF_5\). Now the linear 
	epimorphism \(\psi\colon \FF_5^n\lra \FF_5^2\) defined by~\({\psi(h+\xi v)=(\phi(h), \xi)}\)
	for all \(h\in H\) and \(\xi\in\FF_5\) satisfies 
	\[
		|\psi(A)|=|\phi(X)|+|\phi(Y)|+|\phi(Z)|=6
	\] 
	and Lemma~\ref{lem:62} implies the normality of \(A\). 
    
	\medskip

	{\it \hskip2em Second Case: We have \(\{\lambda, \mu, \nu\}=\{0, 1, 2\}\).}
        
	\smallskip

	Since \(A_0\) is a sum-free subset of \(H\), Fact~\ref{f:23} tells 
	us \(|A_0|\le 2\cdot 5^{n-2}\), whence 
	\[
		2(2|A_1|+|A_2|)+(|A_0|+2|A_2|)
		=
		4|A|-3|A_0|
		\ge 
		(4\cdot 5.6-3\cdot 2)5^{n-2}
		>
		3\cdot (5.2\cdot 5^{n-2})\,.
	\]
	This shows that
	\[
		2|A_1|+|A_2|>5.2\cdot 5^{n-2}
		\quad \text{ or } \quad 
		|A_0|+2|A_2|>5.2\cdot 5^{n-2}\,.
	\]
	
	\begin{claim}\label{clm:1204}
		If \(2|A_1|+|A_2|>5.2\cdot 5^{n-2}\), then \(A\) is normal.
	\end{claim}
	
	\begin{proof}	
		Arguing as in the first case, Corollary~\ref{cor:1135} leads to a linear 
		epimorphism \(\psi\colon \FF_5^n\lra\FF_5^2\)
		such that \(\ker(\psi)\le H\) and \(2|\psi(A_1)|+|\psi(A_2)|=6\). As in the 
		proof of Lemma~\ref{lem:62} one shows \((0, 0)\not\in \psi(A)\) 
		and \(P\pm Q\not\in \psi(A)\) for all distinct \(P, Q\in\psi(A_1\cup A_2)\).  
		So if either \(|\psi(A_1)|\ge 3\) or \(|\psi(A_2)|\ge 3\), 
		then \(\psi(A)\cap \psi(H)=\vn\)
		and we reach the contradiction~\({A_0=\vn}\). 
		In other words, we have \(|\psi(A_1)|=|\psi(A_2)|=2\) and \(|\psi(A_0)|\le 2\) 
		follows. Now \(|\psi(A)|\le 6\) and Lemma~\ref{lem:62} imply that \(A\) is normal.
	\end{proof} 
	
	It remains to discuss the case \(|A_0|+2|A_2|>5.2\cdot 5^{n-2}\). As before, 
	we find a linear epimorphism \(\psi\colon \FF_5^n\lra\FF_5^2\)
	satisfying \(\ker(\psi)\le H\) and \(|\psi(A_0)|+2|\psi(A_2)|=6\), and observe that 
	only the case \(|\psi(A_0)|=|\psi(A_2)|=2\) is possible. This yields 
	\[
		2|A_1|+|A_2|=2|A|-2|A_0|-|A_2|\ge (2\cdot 5.6-2\cdot 2-2)5^{n-2}=5.2\cdot 5^{n-2}
	\]
	and, unless this holds with equality, Claim~\ref{clm:1204} implies the normality 
	of \(A\). So we can suppose \(|A_0|=2\cdot 5^{n-2}\), which means that
	\(\psi^{-1}(P)\subseteq A\) holds for both points \(P\in \psi(A_0)\). 
	As this leads to \(|\psi(A_1)|\le 2\), we again have \(|\psi(A)|\le 6\) 
	and \(A\) is normal by Lemma~\ref{lem:62}.   
\end{proof}

\subsection{VL-sets}
The goal of this subsection is to establish a connection between the three special 
fishy functions~\(f_\alpha\), \(f_\beta\), and \(f_\gamma\) appearing in 
Proposition~\ref{prop:fish} and \(\VL\)-sets. We commence with the three dimensional
case. 

\begin{lemma}\label{lem:3dim}
	 Given a sum-free set \(A\subseteq \FF_5^3\), 
	 let \(I_{x, y}=\{z\in\FF_5\colon (x, y, z)\in A\}\) for every planar 
	 point \((x, y)\in\FF_5^2\). If there exists a letter \(\tau\in\{\alpha, \beta, \gamma\}\)
	 such that \(|I_{x, y}|=f_\tau(x, y)\) holds for every \((x, y)\in\FF_5^2\), 
	 then \(A\) is isomorphic to \(\Lambda\). 
\end{lemma}

\begin{proof}
	Beginning with the case \(\tau=\alpha\) we need to study the two-element 
	sets \(I_{-1, 2}\), \(I_{2, 1}\), \(I_{1, -2}\), and \(I_{-2, -1}\). 
	The main point is that 
	\[
		U=\{(x, -2x, y)\colon x, y\in \FF_5\}
	\] 
	is a two-dimensional subspace of \(\FF_5^3\) and \(A\cap U\) is a sum-free subset 
	of \(U\) of size \(8\). Thus there exists a line \(L\subseteq U\) such 
	that \(A\cap U\subseteq L\cup (-L)\). 
	As \(L\) cannot be parallel to \(L_{0, 0}\), we can assume without loss of 
	generality that \(L=\{(x, -2x, 2x+2)\colon x\in \FF_5\}\), 
	which leads to
    \[
    	I_{-1, 2}=\{0, 1\}\,, \quad
		I_{2, 1}=\{1, 2\}\,, \quad
		I_{1, -2}=\{-1, 0\}\,, \quad
		I_{-2, -1}=\{-2, -1\}\,,
	\]
	and~\(A=\Lambda\).
	
	Next we consider the case \(\tau=\beta\), i.e., \(|I_{x, y}|=f_\beta(x, y)\) for 
	all \((x, y)\in \FF_5^2\). The missing element of \(I_{1, 0}\) can be 
	supposed to be \(0\). For each point \((x, y)\in \FF_5 \times \FF_5^\times\)
	we have \((I_{x, y}+I_{1, 0})\cap I_{x+1, y}=\vn\) and, 
	consequently, \(I_{x, y}=I_{x+1, y}\). 
	In other word, there exists for every \(y\in \FF_5^\times\) a constant~\({c_y\in \FF_5}\) 
	such that \(I_{x, y}=\{c_y\}\) holds for every \(x\in \FF_5\). 
	Due to \((I_{0, y}+I_{1, -y})\cap I_{1, 0}=\vn\) we have \(c_{-y}=-c_y\) for 
	every \(y\in \FF_5^\times\). An appropriate automorphism of \(\FF_5^3\) allows us to 
	suppose \(c_1=c_{-1}=0\) and \(c_2\in \{0, 1, 2\}\). The case \(c_2=0\) is impossible 
	due to \((I_{0,1}+I_{0, 1})\cap I_{0, 2}=\vn\). 
	Finally, \(I_{-1, 0}\cap \{0\}=(I_{-1, 0}+I_{1, 1})\cap I_{0, 1}=\vn\) 
	yields \(I_{-1, 0}=\FF_5^\times\). 
	
	Now the structure of \(A\) becomes more transparent when we project this set not
	into the \mbox{\(x\)-\(y\)-plane} but rather into the \(x\)-\(z\)-plane. That is, 
	we consider the function \(h\colon \FF_5^2\lra\RR_{\ge 0}\) defined by 
	\(h(x, z)=|\{y\in \FF_5\colon (x, y, z)\in A\}|\) for all \((x, z)\in \FF_5^2\). 
	For both cases of \(c_2\), this function is displayed in Figure~\ref{fig:c2}.  
	It is now immediate that \(h\) is isomorphic to \(f_\alpha\), whence \(A\) is isomorphic 
	to \(\Lambda\). 
	
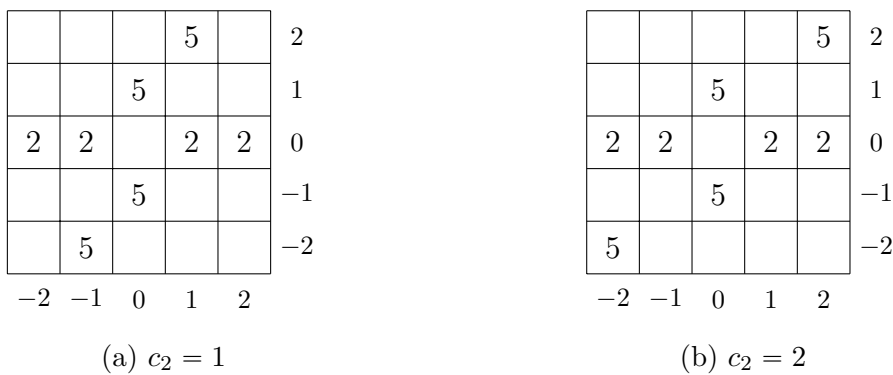
\begin{figure}[h]
	\centering
\hspace{4em}
	\begin{subfigure}[b]{.32\textwidth}
	\centering
		\begin{tikzpicture}[scale=0.7]
		\foreach \x in {-2,...,3} \draw (-2.5, \x-.5)--(2.5, \x-.5) (\x-.5, -2.5)--(\x-.5, 2.5);
		\foreach \x in {-2,...,2} 
			\node at (\x,0) [label={[yshift=-2.5cm]\footnotesize\(\x\)}] {};
		\foreach \y in {-2,...,2} 
			\node at (0,\y) [label={[xshift=2.1cm, yshift=-0.4cm]\footnotesize\(\y\)}] {};
		\foreach \x/\y in {0/1,0/-1,1/2,-1/-2} \draw (\x,\y) node{\(5\)};
		\foreach \x in {-2,-1,1,2} \draw (\x,0) node{\(2\)};
		\end{tikzpicture}
	\caption{\(c_2=1\)}\label{fig:c21}
	\end{subfigure}
\hfill
	\begin{subfigure}[b]{.32\textwidth}
	\centering
		\begin{tikzpicture}[scale=0.7]
		\foreach \x in {-2,...,3} \draw (-2.5, \x-.5)--(2.5, \x-.5) (\x-.5, -2.5)--(\x-.5, 2.5);
		\foreach \x in {-2,...,2} 
			\node at (\x,0) [label={[yshift=-2.5cm]\footnotesize\(\x\)}] {};
		\foreach \y in {-2,...,2} 
			\node at (0,\y) [label={[xshift=2.1cm, yshift=-0.4cm]\footnotesize\(\y\)}] {};
		\foreach \x/\y in {0/1,0/-1,2/2,-2/-2} \draw (\x,\y) node{\(5\)};
		\foreach \x in {-2,-1,1,2} \draw (\x,0) node{\(2\)};
		\end{tikzpicture}
	\caption{\(c_2=2\)}\label{fig:c22}
	\end{subfigure}
\hspace{4em}
\caption{The function \(h\)}\label{fig:c2}
\end{figure}	 
	
	Let us finally address the case \(\tau=\gamma\). Arguing as in the previous case, we can 
	suppose \(I_{1, 0}=I_{0, 1}=\FF_5^\times\) and prove the existence of some \(a\in \FF_5\)
	such that \(I_{x, y}=\{a\}\) holds for the twelve points \((x, y)\in \FF_5^2\) 
	with \(f_\gamma(x, y)=1\). Due to $(I_{-1,-2}+I_{2,2})\cap I_{1,0}=\vn$ we have \(a=0\) 
	but now the contradiction $(I_{2,1}+I_{1,2})\cap I_{-2,-2}\ne\vn\) arises.
	Thus the last case is impossible. 
\end{proof}

Our next goal is to extend this result to higher dimensions. 
As explained in~\S\ref{sec:overview} with every sum-free 
set \(A\subseteq \FF_5^n\) and every linear epimorphism 
\(\phi\colon \FF_5^n\lra\FF_5^2\) we associate the function 
\[
	f^A_\phi\colon \FF_5^2\lra \RR_{\ge 0}\,, \qquad 
	P\longmapsto \frac{|A\cap \phi^{-1}(P)|}{5^{n-3}}\,.
\]
It will also be convenient to set \(A^\phi_P=A\cap \phi^{-1}(P)\) 
for every \(P\in \FF_5^2\). We would like to point out an obvious 
consequence of Corollary~\ref{cor:1135}

\begin{cor}\label{lem:1500}
	Let a sum-free set \(A\subseteq \FF_5^n\) and  a linear 
	epimorphism \(\phi\colon \FF_5^n\lra \FF_5^2\) be given. 
	If for some \(P, Q, R\in \FF_5^2\) and \(\eps\in \{-1, 1\}\) with \(R=P+\eps Q\)
	we have 
	\begin{itemize}
		\item \(f^A_\phi(P), f^A_\phi(Q), f^A_\phi(R)\in \{1, 2, 3, 4\}\) 
		\item and \(f^A_\phi(P)+f^A_\phi(Q)+f^A_\phi(R)=6\), 
	\end{itemize}
	then \(K=\Sym(A^\phi_P)=\Sym(A^\phi_Q)=\Sym(A^\phi_R)\) 
	is a \((n-3)\)-dimensional subspace of \(\ker(\phi)\) and for 
	each \(X\in \{P, Q, R\}\) the set~\(A^\phi_X\) 
	is a union of \(f^A_\phi(X)\) translates of \(K\). \qed
\end{cor}

\begin{lemma}\label{lem:VL}
	Let \(A\subseteq \FF_5^n\) be sum-free. 
    If there exists a linear epimorphism \(\phi\colon \FF_5^n\lra \FF_5^2\) such 
    that \(f^A_\phi=f_\tau\) holds for some letter \(\tau\in\{\alpha, \beta, \gamma\}\), 
    then \(A\) is a \(\VL\)-set. 
\end{lemma}

\begin{proof}
	We commence with the case \(\tau=\alpha\), where our main task is to analyse the 
	sets \(A^\phi_{-1, 2}\), \(A^\phi_{2, 1}\), \(A^\phi_{1, -2}\), and \(A^\phi_{-2, -1}\). 
	Since 
	\[
		(-1, 2)+(2, 1)=(1, -2)\,, \qquad (2, 1)+(1, -2)=(-2, -1)\,,
	\] 
	and 
	\[
		f^A_\phi(-1, 2)+f^A_\phi(2, 1)+f^A_\phi(1, -2)
		=
		f^A_\phi(2, 1)+f^A_\phi(1, -2)+f^A_\phi(-2, -1)
		=
		6\,,
	\]
	Corollary~\ref{lem:1500} tells us that all four of them have the same symmetry set 
	\[
		K=\Sym(A_{-1, 2})=\Sym(A_{2, 1})=\Sym(A_{1, -2})=\Sym(A_{-2, -1})\,,
	\]
	which is a hyperplane in~\(\ker(\phi)\). Furthermore, each of our four sets is a 
	union two translates of \(K\). 
	Altogether, \(A\) is a union of \(28\) translates of \(K\) and thus 
	isomorphic to a set of the form~\({B\times \FF_5^{n-3}}\), where 
	\(B\subseteq \FF_5^3\) is a sum-free set satisfying the hypothesis of 
	Lemma~\ref{lem:3dim} for~\({\tau=\alpha}\). So~\(B\) is isomorphic to~\(\Lambda\) 
	and~\(A\) is a \(\VL\)-set. 
	
	The cases \(\tau=\beta\) and \(\tau=\gamma\) are similar.
\end{proof}

\subsection{Fishy projections}
The next result connects the study of sum-free sets to our work on fishy functions.

\begin{lemma}\label{lem:1922}
	For every sum-free set \(A\subseteq \FF_5^n\) of size \(|A|\ge 28\cdot 5^{n-3}\) 
	and every linear epimorphism \(\phi\colon \FF_5^n\lra\FF_5^2\) the 
	function \(f^A_\phi\) is fishy.
\end{lemma}

\begin{proof}
	Set \(f=f^A_\phi\). Property~\ref{it:fish1} follows from 
	\[
		\|f\|_1=\frac{|A|}{5^{n-3}}\ge 28
	\]
	and from the fact that  
	\[
		f(P)\le \frac{|\phi^{-1}(P)|}{5^{n-3}}=5
	\]
	is fulfilled for every \(P\in\FF_5^2\). Since for every \(X\subseteq \FF_5^2\) 
	the number \(5^{n-3}f(X)\) is an integer,~\ref{it:fish2} holds as well.  
	
	Proceeding with~\ref{it:fish3} we consider two points \(P, Q\in\FF_5^2\) and a sign 
	\(\eps\in\{-1, 1\}\) such that~\(P\),~\(Q\), and~\(R=P+\eps Q\) are in the support 
	of \(f\). Setting \(t=5\lceil f(R)\rceil-5\) we 
	have 
	\[
		|A^\phi_R|=5f(R)\cdot 5^{n-4}>t\cdot 5^{n-4}\,.
	\]
	Since \(t\) is divisible by \(5\), Lemma~\ref{lem:1419} yields 
	\[
		|A^\phi_P|+|A^\phi_Q|\le (25-t)5^{n-4}=(6-\lceil f(R)\rceil)5^{n-3}\,,
	\]
	whence \(f(P)+f(Q)\le 6-\lceil f(R)\rceil\). 
\end{proof}

\begin{proof}[Proof of Proposition~\ref{prop:0133}]
	Let \(\phi\colon \FF_5^n\lra\FF_5^2\) be a linear epimorphism mapping \(T\) to some 
	point \(P\in \FF_5^2\). By Lemma~\ref{lem:1922} the function \(f=f^A_\phi\) is fishy. 
	Since 
	\[
		\|f\|_\infty\ge f(P)=\frac{|A\cap T|}{5^{n-3}}>3\,,
	\]
	Proposition~\ref{prop:fish} tells us that either the support of \(f\) 
	can be covered by three parallel lines or~\(f\) is isomorphic to one 
	of the functions \(f_\alpha\), \(f_\beta\), or \(f_\gamma\). In the 
	former case, \(A\) can be covered by three parallel hyperplanes and 
	Lemma~\ref{lem:3planes} informs us that \(A\) is normal. In the latter 
	case, Lemma~\ref{lem:VL} applies and \(A\) is a VL-set. 
\end{proof}

\begin{proof}[Proof of Proposition~\ref{prop:0135}]
	Fix a linear epimorphism \(\phi\colon \FF_5^n\lra\FF_5^2\) sending~\(T\) 
	to some point \(P_\star\in \FF_5^2\) and, consequently,~\(-T\) to~\(-P_\star\). 
	By Lemma~\ref{lem:1922} the function \(g=f^A_\phi\) is fishy and our hypothesis 
	on \(T\) translates to 
	\begin{equation}\label{eq:0149}
		g(P_\star)+g(-P_\star)=\frac{|A\cap T|+|A\cap (-T)|}{5^{n-3}}>5.6\,,
	\end{equation}
	so we can suppose without loss of generality that \(g(P_\star)>2.8\). 
	If \(\|g\|_\infty>3\), then Proposition~\ref{prop:0133} implies immediately that \(A\)
	has the desired form, so it suffices to treat the case \(\|g\|_\infty\le 3\) 
	in the sequel. Now~\eqref{eq:0149} implies \(g(-P_\star)>5.6-3>2.5\). 
	
	\begin{claim}\label{clm:1312}
		For every point \(Q\in \FF_5^2\) such that \(g(Q), g(P_\star+Q)>0\) 
		and \(g(Q)+g(P_\star+Q)>2.2\) there exists a hyperplane \(K_Q\le \ker(\phi)\)
		such that each of \(A^\phi_{P_\star}\) and \(A^\phi_Q\cup A^\phi_{P_\star+Q}\)  
		can be covered by three translates of \(K_Q\). 
	\end{claim}
	
	\begin{proof}
		We pick arbitrary points \(v_{P_\star}\in \phi^{-1}(P_\star)\) 
		and \(v_Q\in\phi^{-1}(Q)\).
		Now Lemma~\ref{lem:1419} applied to~\(\ker(\phi)\) instead of \(\FF_5^\ell\),
		the sets, 
		\[
			X=A^\phi_{P_\star+Q}-(v_{P_\star}+v_Q)\,, \quad
		 	Y=v_Q-A^\phi_Q\,, \quad  
		 	Z=A^\phi_{P_\star}-v_{P_\star}\,,
		\] 
		and to \(t=14\) yields the desired hyperplane \(K_Q\). 
	\end{proof}
	
	Because of \(g(P_\star)>2.8\) and \(\|g\|_\infty\le 3\) Lemma~\ref{lem:314} shows that 
	the assumption of Claim~\ref{clm:1312} is satisfied for at least one point \(Q\) and, 
	consequently, there exists a hyperplane \(K\le \ker(\phi)\) such that \(A^\phi_{P_\star}\)
	is contained in the union of three translates of \(K\). Assume for the sake of contradiction 
	that \(K\) is not uniquely determined by this property. Since any two nonparallel affine 
	hyperplanes in \(\ker(\phi)\) intersect in a \((n-4)\)-dimensional affine subspace 
	of \(\ker(\phi)\), this yields \(|A^\phi_{P_\star}|\le 3\cdot 3\cdot 5^{n-4}\) and we 
	get the contradiction \(g(P_\star)\le 1.8\). 
	
	This proves \(K_Q=K\) for all points \(Q\in\FF_5^2\) satisfying the hypothesis 
	of Claim~\ref{clm:1312}. Let \(\pi\colon \FF^3_5\lra \FF_5^2\) be the projection 
	to the first two coordinates, i.e., the map \((x, y, z)\longmapsto (x, y)\), choose  
	a linear epimorphism  \(\psi\colon \FF_5^n\lra \FF_5^3\) with kernel~\(K\)
	such that \(\phi=\pi\circ\psi\), and define  
	\[
		f\colon \FF_5^3\lra\RR_{\ge 0}  
		\quad \text{ by } \quad 
		f(P)=\frac{|A\cap \psi^{-1}(P)|}{5^{n-3}}
	\] 
	for every \(P\in\FF_5^3\). Owing to \(\phi=\pi\circ\psi\) we have 
	\(\psi^{-1}(L_P)=\phi^{-1}(P)\) for every \(P\in\FF_5^2\) and, 
	therefore, \(g\) is the standard projection of~\(f\).
	
	We contend that \(f\) is acceptable. The properties~\ref{it:owl2} 
	and~\ref{it:owl4}\,--\,\ref{it:owl6} have already been addressed. 
	Moreover,~\ref{it:owl1} is clear and~\ref{it:owl3} follows easily 
	from~Corollary~\ref{Bdumm}.
	
	Now Proposition~\ref{prop:klein} delivers a line \(L\subseteq\FF_5^3\)
	such that \(f(L)>3\). The affine subspace~\({U=\psi^{-1}(L)}\) of \(\FF_5^n\)
	has codimension~\(2\) and satisfies \(|A\cap U|=f(L)\cdot 5^{n-3}>3\cdot 5^{n-3}\).
	Owing to Proposition~\ref{prop:0133} it follows that \(A\) is either normal 
	or a \(\VL\)-set.  
\end{proof}

\section{Concluding remarks}\label{sec:conclude}

The notation for sum-free sets employed in the introduction 
suggests the following hierarchy of sum-free sets introduced 
in~\cite{RZ24a}*{Definition~1.2}. 

\begin{definition}\label{dfn:traum}
    Let \(G\) be a finite abelian group. Starting with 
    \[
    	\SFR_0(G)=\{A\subseteq G\colon A \text{ is sum-free}\}
	\]
	we define by recursion on \(k\) the numbers and sets 
    \begin{itemize}
    	\item \(\sfr_k(G)=\max\{|A|\colon A\in \SFR_k(G)\}\)\,,
        \item \(\SFRR_k(G)=\{A\in \SFR_k(G)\colon |A|=\sfr_k(G)\}\)\,,   	
    	\item \(\SFR_{k+1}(G)=\{A\in \SFR_{k}(G)\colon \text{there
				is no \(B\in\SFRR_k(G)\) with \(A\subseteq B\)}\}\)\,.
	\end{itemize}
\end{definition}

This is to some extent inspired by a similarly defined hierarchy for intersecting set 
systems, which has a vast literature (see e.g.,~\cites{EKR, HK, HM, Joanna}). As we have 
already pointed out, \(\SFRR_0(G)\) is known for every finite abelian group \(G\) 
(see~\cite{Ba16}). Research on \(\sfr_1(G)\) has so far mainly focused on vector spaces 
over finite fields. In fact, prior to this work \(\sfr_1(\FF_p^n)\) and even \(\SFRR_1(\FF_p^n)\) 
were known for all~\(p\ne 5\), so that Theorem~\ref{thm:main} solves the last open problem 
of its kind. It would of course be interesting if one could determine \(\sfr_1(G)\) 
for broader classes of abelian groups as well. 

For \(k\ge 2\) some discussion on \(\sfr_k(\FF_p^n)\) can be found in~\cites{RZ24a, R24c}.
Here we would only like to point out that if \(p\) denotes a prime number with \(p\ge 11\)
and \(p\equiv -1\pmod{3}\), then for \(n\ge 2\) there are sets \(A\in \SFRR_1(\FF_p^n)\)
such that there are two distinct vectors \(x, y\in \FF_p^n\) not belonging to 
\[
	A\cup (A+A)\cup (A-A)\,.
\]   
Thus the only reason why \(A\cup\{x\}\) and \(A\cup\{y\}\) fail to be sum-free is that
\(2x\) and \(2y\) are in \(A\). In the situations we have in mind \(x+y\in A\) 
and \(x-y\not\in A\) hold  as well, and the sum-free 
set \(A\cup \{x, y\}\sm \{2x, x+y, 2y\}\) exemplifies 
\(\sfr_2(\FF_p^n)=\sfr_1(\FF_p^n)-1\), which is, perhaps, a bit underwhelming. 
The same kind of construction shows \(\sfr_1(\FF_p^n)=\sfr_0(\FF_p^n)-1\) 
whenever \(p\equiv 1\pmod{3}\).

For \(p=2\) and \(p=3\), on the other hand, the entire hierarchy 
\(\bigl(\sfr_k(\FF_p^n)\bigr)_{k\in\NN_0}\) is interesting (see~\cites{DT, R24c}). 
It is not difficult to check \(\Lambda\cup (\Lambda+\Lambda)\cup (\Lambda-\Lambda)=\FF_5^3\)
and thus there is some reason to believe that the investigation of \(\sfr_2(\FF_5^n)\) 
will lead to new insights as well. 

The following related group invariant has been introduced by Vsevolod Lev~\cite{VL05}. 
Call a subset~\(X\) of an abelian group \(G\) {\it aperiodic} if \(\Sym(X)=\{0\}\). 
If there exists an aperiodic, inclusion-wise maximal sum-free subset of \(G\), 
we denote the largest size that such a set can have by \(t(G)\); otherwise, we simply 
set \(t(G)=0\). Whenever \(p\ne 5\) the numbers \(t(\FF_p^n)\) are known  
(see~\cites{DT, VL05, R24c}), while our understanding of the case \(p=5\) is very limited. 
Clearly, we have \(t(\FF_5)=0\) and Theorem~\ref{5n=2} shows \(t(\FF_5^2)=5\). 
Moreover, the aperiodicity of \(\Lambda\) implies~\({t(\FF_5^3)=28}\), but for \(n\ge 4\) 
the determination of \(t(\FF_5^4)\) is open.

\subsection*{Acknowledgement} 
The second author would like to thank Leo Versteegen for introducing 
her to this topic and some early discussions.

\end{document}